\documentclass[11pt,sumlimits,intlimits]{amsart}
\usepackage{url,subfig,graphicx,cite,hyperref}
\usepackage[hmargin=1.5in]{geometry}

\theoremstyle{plain}
\newtheorem{theorem}[subsection]{Theorem}
\newtheorem*{theorem*}{Theorem}

\newtheorem*{question*}{Question}

\newtheorem{lemma}[subsection]{Lemma}
\newtheorem*{lemma*}{Lemma}
\newtheorem{proposition}[subsection]{Proposition}
\newtheorem*{proposition*}{Proposition}
\theoremstyle{definition}
\newtheorem{definition}[subsection]{Definition}
\newtheorem*{definition*}{Definition}
\newtheorem{remark}[subsection]{Remark}

\newtheorem{notation}[subsection]{Notation}

\newtheorem*{example*}{Example}
\newtheorem*{notation*}{Notation}

\newtheoremstyle{myclaim}
  {1ex}
  {1ex}
  {\it}
  {\parindent}
  {\it}
  {.}
  { }
  {}
\theoremstyle{myclaim}

\newtheorem*{claim*}{Claim}

\newtheoremstyle{note}
  {}
  {}
  {}
  {}
  {\normalfont}
  {.}
  {.5em}
  {}

\newtheoremstyle{citing}
  {}
  {}
  {\itshape}
  {}
  {\bfseries}
  {.}
  {.5em}
  {\thmnote{#3}}

\theoremstyle{citing}

\newcommand{\Z}{\ensuremath{\mathbb{Z}}}

\newcommand{\R}{\ensuremath{\mathbb{R}}}
\newcommand{\C}{\ensuremath{\mathbb{C}}}
\newcommand{\T}{\ensuremath{\mathbb{T}}}
\newcommand{\F}{\ensuremath{\mathbb{F}}}

\newcommand{\norm}[1]{\ensuremath{\| #1 \|}}

\renewcommand{\epsilon}{\varepsilon}
\providecommand{\cal}{\ensuremath{\mathcal}}
\DeclareMathOperator{\id}{id}

\makeatletter
\newcommand*{\centerfloat}{%
  \parindent \z@
  \leftskip \z@ \@plus 1fil \@minus \textwidth
  \rightskip\leftskip
  \parfillskip \z@skip}
\makeatother

\let\originalleft\left
\let\originalright\right
\renewcommand{\left}{\mathopen{}\mathclose\bgroup\originalleft}
\renewcommand{\right}{\aftergroup\egroup\originalright}

\DeclareSymbolFont{EUEX}{U}{euex}{m}{n}

\DeclareSymbolFont{euexlargesymbols}{U}{euex}{m}{n}
\DeclareMathSymbol{\intop}{\mathop}{euexlargesymbols}{"52}
     \def\int{\intop\nolimits}

\DeclareSymbolFont{euexsymbols}     {U}{euex}{m}{n}
\DeclareMathSymbol{\smallint}{\mathop}{euexsymbols}{"52}

\DeclareMathOperator{\tr}{tr}

\DeclareMathOperator{\Aut}{Aut}

\DeclareMathOperator{\End}{End}
\DeclareMathOperator{\ch}{ch}
\DeclareMathOperator{\GL}{GL}
\DeclareMathOperator{\Td}{Td}
\newcommand{\del}{\ensuremath{\nabla}}

\begin{document}

\title[Quasi-representations of surface groups]{%
  Quasi-representations of surface groups}

\author{Jos\'e R. Carri\'on}
\email{jcarrion@math.purdue.edu}

\author{Marius Dadarlat}
\email{mdd@math.purdue.edu}

\thanks{M.D. was partially supported by NSF grant \#DMS--1101305.}
\thanks{J.R.C. was partially supported by a Purdue Research Foundation
  grant.}

\address{%
  Department of Mathematics,
  Purdue University,
  West Lafayette, IN, 47907,
  United States}

\date{June 15, 2013.}

\begin{abstract}
  By a quasi-representation of a group $G$ we mean an approximately
  multiplicative map of $G$ to the unitary group of a unital
  $C^*$-algebra.  A quasi-representation induces a partially defined
  map at the level $K$-theory.
  
  In the early 90s Exel and Loring associated two invariants to
  almost-commuting pairs of unitary matrices $u$ and $v$: one a
  $K$-theoretic invariant, which may be regarded as the image of the
  Bott element in $K_0(C(\mathbb{T}^2))$ under a map induced by
  quasi-representation of $\mathbb{Z}^2$ in $U(n)$; the other is the
  winding number in $\mathbb{C}\setminus \{0\}$ of the closed path
  $t\mapsto \det(tvu + (1-t)uv)$.  The so-called Exel-Loring formula
  states that these two invariants coincide if $\|uv - vu\|$ is
  sufficiently small.

  A generalization of the Exel-Loring formula for
  quasi-representations of a surface group taking values in $U(n)$ was
  given by the second-named author.  Here we further extend this
  formula for quasi-representations of a surface group taking values
  in the unitary group of a tracial unital $C^*$-algebra.
\end{abstract}

\maketitle

\section{Introduction}

Let $G$ be a discrete countable group.  In \cite{Dadarlat12,
  Dadarlat11a} the second-named author studied the question of how
deformations of the group $G$ (or of the group $C^*$-algebra $C^*(G)$)
into the unitary group of a (unital) $C^*$-algebra $A$ act on the
$K$-theory of the algebras $\ell^1(G)$ and $C^*(G)$.  By a deformation
we mean an almost-multiplicative map, a \emph{quasi-representation},
which we will define precisely in a moment.  Often, matrix-valued
multiplicative maps are inadequate for detecting the $K$-theory of the
aforementioned group algebras.  In fact, if a countable, discrete,
torsion free group $G$ satisfies the Baum-Connes conjecture, then a
unital finite dimensional representation $\pi\colon C^*(G)\to M_r(\C)$
induces the map $r\cdot \iota_*$ on $K_0(C^*(G))$, where $\iota$ is
the trivial representation of $G$ (see
\cite[Proposition~3.2]{Dadarlat11a}).  It turns out that
almost-multiplicative maps detect $K$-theory quite well for large
classes of groups: one can implement any group homomorphism of
$K_0(C^*(G))$ to $\Z$ on large swaths of $K_0(C^*(G))$ using
quasi-representations (see \cite[Theorem~3.3]{Dadarlat11a}).

Knowing that quasi-representations may be used to detect $K$-theory,
we turn to how it is that they act.  An index theorem of Connes,
Gromov and Moscovici in \cite{Connes-Gromov-etal90} is very relevant
to this topic, in the following context. Let $M$ be a closed
Riemannian manifold with fundamental group $G$ and let $D$ be an
elliptic pseudo-differential operator on $M$.  The equivariant index
of $D$ is an element of $K_0(\ell^1(G))$.  Connes, Gromov and
Moscovici showed that the push-forward of the equivariant index of $D$
under a quasi-representation of $G$ coming from parallel transport in
an almost-flat bundle $E$ over $M$ is equal to the index of $D$
twisted by $E$.

At around the same time, Exel and Loring studied two invariants
associated to pairs of almost-commuting scalar unitary matrices $u,
v\in U(r)$.  One is a $K$-theory invariant, which may be regarded as
the push-forward of the Bott element $\beta$ in the $K_0$-group of
$C(\T^2)\cong C^*(\Z^2)$ by a quasi-representation of $\Z^2$ into the
unitary group $U(r)$.  The Exel-Loring formula proved in
\cite{Exel-Loring91} states that this invariant equals the winding
number in $\C\setminus \{0\}$ of the path $ t\mapsto \det\big( (1-t)uv
+ tvu \big)$.  An extension of this formula for almost commuting
unitaries in a $C^*$-algebra of tracial rank one is due to H.~Lin and
plays an important role in the classification theory of amenable
$C^*$-algebras.  In a different direction, the Exel-Loring formula was
generalized in \cite{Dadarlat12} to finite dimensional
quasi-representations of a surface group using a variant of the index
theorem of \cite{Connes-Gromov-etal90}.
Remark~\ref{rem:exel-loring-orig} discusses the Exel-Loring formula in
more detail.

In \cite{Dadarlat12}, the second-named author used the
Mishchenko-Fomenko index theorem to give a new proof and a
generalization of the index theorem of Connes, Gromov and Moscovici
that allows $C^*$-algebra coefficients.  In this paper we use this
generalization to address the question of how a quasi-representation
$\pi$ of a surface group in the unitary group of a tracial
$C^*$-algebra acts at the level of $K$-theory.  We extend the
Exel-Loring formula to a surface group $\Gamma_g$ (with canonical
generators $\alpha_i$, $\beta_i$) and coefficients in a unital
$C^*$-algebra $A$ with a trace $\tau$.  Briefly, writing
$K_0(\ell^1(\Gamma_g)) \cong \Z [1] \oplus \Z \mu[\Sigma_g]$ we have
\begin{displaymath}
  \tau \big( \pi_\sharp(\mu[\Sigma_g]) \big)
  = \frac{1}{2\pi i} \tau\biggl(
  \log\bigg( \prod_{i=1}^g \big[\pi(\alpha_i),\pi(\beta_i)\big] \biggr)
  \bigg),
\end{displaymath}
where $[1]$ is the $K_0$ class of the unit $1\in \ell^1(\Gamma_g)$,
$[\Sigma_g]$ is the fundamental class in $K$-homology of the genus $g$
surface $\Sigma_g$, and $\mu\colon K_0(\Sigma_g)\to K_0(\ell^1(G))$ is
the $\ell^1$-version of the assembly map of Lafforgue.  For a complete
statement see Theorem~\ref{thm:Mg-formula}.  In the proof we make use
of Chern-Weil theory for connections on Hilbert $A$-module bundles as
developed by Schick \cite{Schick05} and the de la Harpe-Skandalis
determinant \cite{Harpe-Skandalis84} to calculate the first Chern
class of an almost-flat Hilbert module $C^*$-bundle associated to a
quasi-representation (Theorem~\ref{thm:chern-class}).

The paper is organized as follows.  In Section~\ref{sec:main-result}
we define quasi-representations and the invariants we are interested
in, and state our main result, Theorem~\ref{thm:Mg-formula}.  The
invariants make use of the Mishchenko line bundle, which we discuss in
Section~\ref{sec:misch-line-bundle}.  The push-forward of this bundle
by a quasi-representation is considered in
Section~\ref{sec:hilb-module-bundles}.  Section~\ref{sec:chern-class}
contains our main technical result, Theorem~\ref{thm:chern-class},
which computes one of our invariants in terms of the de la
Harpe-Skandalis determinant \cite{Harpe-Skandalis84}.  To obtain the
formula given in the Theorem~\ref{thm:Mg-formula}, we must work with
concrete triangulations of oriented surfaces; this is contained in
Section~\ref{sec:oriented-surfaces}.  Assembling these results in
Section~\ref{sec:proof-main-result} yields a proof of
Theorem~\ref{thm:Mg-formula}.

\section{The main result}
\label{sec:main-result}

In this section we state our main result.  It depends on a result in
\cite{Dadarlat12} that we revisit.  Let us provide some notation and
definitions first.

Let $G$ be a discrete countable group and $A$ a unital $C^*$-algebra.

\begin{definition}
  Let $\epsilon > 0$ and let $\mathcal{F}$ be a finite subset of $G$.
  An \emph{$(\mathcal{F}, \epsilon)$-representation} of $G$ in $U(A)$
  is a function $\pi \colon G\to U(A)$ such that for all $s, t\in
  {\cal F}$ we have
  \begin{itemize}
  \item $\pi(1) = 1$,
  \item $\norm{\pi(s^{-1}) - \pi(s)^*} < \epsilon$, and
  \item $\norm{\pi(st) - \pi(s)\pi(t)} < \epsilon$.
  \end{itemize}
  We refer to the third condition by saying that $\pi$ is
  $(\mathcal{F}, \epsilon)$-multiplicative.  Let us note that the
  second condition follows from the other two if we assume that
  $\mathcal{F}$ is symmetric, i.e. $\mathcal{F}=\mathcal{F}^{-1}$.  A
  \emph{quasi-representation} is an $(\mathcal{F},
  \epsilon)$-representation where $\mathcal{F}$ and $\epsilon$ are not
  necessarily specified.
\end{definition}

A quasi-representation $\pi\colon G\to U(A)$ induces a map (also
denoted $\pi$) of the Banach algebra $\ell^1(G)$ to $A$ by $\sum
\lambda_s s \mapsto \sum \lambda_s \pi(s)$.  This map is a unital
linear contraction.  We also write $\pi$ for the extension of $\pi$ to
matrix algebras over $\ell^1(G)$.

\subsection{Pushing-forward via quasi-representations}
\label{subsec:pushing-forward-via}
A group homomorphism $\pi\colon G\to U(A)$ induces a map $\pi_*\colon
K_0(\ell^1(G))\to K_0(A)$ (via its Banach algebra extension).  We
think of a quasi-representation $\pi$ as inducing a partially defined
map $\pi_\sharp$ at the level of $K$-theory, in the following sense.
If $e$ is an idempotent in some matrix algebra over $\ell^1(G)$ such
that $\|\pi(e) - \pi(e)^2\| < 1/4$, then the spectrum of $\pi(e)$ is
disjoint from the line $\{\text{Re}\, z = 1/2\}$.  Writing $\chi$ for
the characteristic function of $\{\operatorname{Re} z > 1/2\}$, it follows
that $\chi(\pi(e))$ is an idempotent and we set
\begin{displaymath}
  \pi_\sharp(e) = \big[\chi(\pi(e))\big] \in K_0(A).
\end{displaymath}
For an element $x$ in $K_0(\ell^1(G))$ we make a choice of idempotents
$e_0$ and $e_1$ in some matrix algebra over $\ell^1(G)$ such that $x =
[e_0] - [e_1]$.  If $\|\pi(e_i) - \pi(e_i)^2\| < 1/4$ for $i\in \{0,
1\}$, write $\pi_\sharp(x) = \pi_\sharp(e_0) - \pi_\sharp(e_1)$.  The
choice of idempotents is largely inconsequential: given two choices of
representatives one finds that if $\pi$ is multiplicative enough, then
both choices yield the same element of $K_0(A)$.

Of course, the more multiplicative $\pi$ is, the more elements of
$K_0(\ell^1(G))$ we can push-forward into $K_0(A)$.

\subsection{An index theorem}
\label{subsec:an-index-theorem}

Fix a closed oriented Riemannian surface $M$ and let $G$ be its
fundamental group.  Fix also a unital $C^*$-algebra $A$ with a tracial
state $\tau$.  Write $K_0(M)$ for $KK(C(M), \C)$.  Because the
assembly map $\mu\colon K_0(M)\to K_0(\ell^1(G))$ is known to be an
isomorphism in this case \cite{Lafforgue02}, we have
\begin{displaymath}
  K_0(\ell^1(G)) \cong \Z[1] \oplus \Z\mu[M]
\end{displaymath}
where $[M]$ is the fundamental class of $M$ in $K_0(M)$
\cite[Lemma~7.9]{Bettaieb-Matthey-etal05} and $[1]$ is the class of
the unit of $\ell^1(G)$.  Since we are interested in
how a quasi-representation of $G$ acts on $K_0(\ell^1(G))$, we would
like to study push-forward of the generator $\mu[M]$ by a
quasi-representation.

\subsubsection{}

Consider the universal cover $\widetilde{M}\to M$ and the diagonal
action of $G$ on $\widetilde{M}\times \ell^1(G)$ giving rise to the
so-called Mishchenko line bundle $\ell$, $\widetilde{M}\times_G
\ell^1(G)\to M$.  We will discuss it in more detail in
Section~\ref{sec:misch-line-bundle}, where we will give a description
of it as the class of a specific idempotent $e$ in some matrix algebra
over $C(M)\otimes \ell^1(G)$.

If $\pi$ is a quasi-representation of $G$ in $U(A)$, then
$\id_{C(M)}\otimes \pi$ is an almost-multiplicative unital linear
contraction on $C(M)\otimes \ell^1(G)$ with values in $C(M)\otimes A$.
Assuming that $\pi$ is sufficiently multiplicative, we may define the
push-forward of the idempotent $e$ by $\id_{C(M)}\otimes \pi$, just as
in Section~\ref{subsec:pushing-forward-via}.  We set
\begin{displaymath}
  \ell_\pi := (\id_{C(M)}\otimes \pi)_\sharp (\ell) :=
  (\id_{C(M)}\otimes \pi)_\sharp (e) \in
  K_0(C(M)\otimes A).
\end{displaymath}
Let $D$ be an elliptic operator on $M^n$ and let $\mu[D]\in
K_0(\ell^1(G))$ be its image under the assembly map.  Let $q_0$ and
$q_1$ be idempotents in some matrix algebra over $\ell^1(G)$ such that
$\mu[D] = [q_0] - [q_1]$ and write $\pi_\sharp( \mu[D] ) :=
\pi_\sharp(q_0) - \pi_\sharp(q_1)$.  By
\cite[Corollary~3.8]{Dadarlat12}, if $\pi \colon G\to A$ is
sufficiently multiplicative, then
\begin{equation}
  \label{eqn:MFi}
  \tau\big(\pi_\sharp(\mu[D])\big)
  = (-1)^{n(n+1)/2} \big\langle (p_!\ch(\sigma(D))\cup
  \Td(TM\otimes \mathbb{C})\cup \ch_\tau(\ell_\pi) , [M] \big\rangle,
\end{equation}
where $p\colon TM\to M$ is the canonical projection, $\ch(\sigma(D))$
is the Chern character of the symbol of $D$, $\Td(T_\C M)$ is the Todd
class of the complexified tangent bundle, and $[M]$ is the fundamental
homology class of $M$.  Set $\alpha=(-1)^{n(n+1)/2}
(p_!ch(\sigma(D))\cup \Td(TM\otimes \mathbb{C})$.  Then
\eqref{eqn:MFi} becomes
\begin{equation*}
  \tau\big(\pi_\sharp(\mu[D])\big)
  = \big\langle \alpha\cup \ch_\tau(\ell_\pi), [M]\big\rangle
  = \big\langle \ch_\tau(\ell_\pi), \alpha \cap [M] \big\rangle,
\end{equation*}
On the other hand, it follows from the Atiyah-Singer index theorem
that the Chern character in homology $\ch \colon K_0(M)\to
H_*(M;\mathbb{Q})$ is given by
\begin{displaymath}
  \ch[D]
  = \big( (-1)^{n(n+1)/2} p_! \ch( \sigma(D) )
    \cup \Td(T_\C M)
    \big) \cap [M]=\alpha \cap [M].  \end{displaymath}
It follows that 
\begin{equation}\label{eq:1}
  \tau\big(\pi_\sharp(\mu[D])\big)
  = \big\langle \ch_\tau(\ell_\pi), \ch[D] \big\rangle
\end{equation}

In the case of surfaces this formula specializes to the following
statement.

\begin{theorem}[cf. {\cite[Corollary~3.8]{Dadarlat12}}]
  \label{thm:index-theorem}
  Let $M$ be a closed oriented Riemannian surface of genus $g$ with
  fundamental group $G$.  Let $q_0$ and $q_1$ be idempotents in some
  matrix algebra over $\ell^1(G)$ such that $\mu[M] = [q_0] - [q_1]$.
  Then there exist a finite subset $\mathcal{G}$ of $G$ and $\omega >
  0$ satisfying the following.

  Let $A$ be a unital $C^*$-algebra with a tracial state $\tau$ and
  let $\pi\colon G\to U(A)$ be a $(\mathcal{G},
  \omega)$-representation.  Write $\pi_\sharp( \mu[M] ) :=
  \pi_\sharp(q_0) - \pi_\sharp(q_1)$.  Then
  \begin{displaymath}
    \tau\big( \pi_\sharp(\mu[M]) \big) =
    \langle \ch_\tau (\ell_\pi), [M] \rangle.
  \end{displaymath}
\end{theorem}
Here $\ch_\tau\colon K_0\big( C(M)\otimes A \big) \to H^{2}(M, \R)$ is
a Chern character associated to $\tau$ (see
Section~\ref{sec:chern-class}), and $[M]\in H_{2}(M, \R)$ is the
fundamental class of $M$.

\begin{proof}
  Given another pair of idempotents $q_0', q_1'$ in some matrix
  algebra over $\ell^1(G)$ such that $\mu[M] = [q_0'] - [q_1']$, there
  is an $\omega_0 > 0$ such that if $0 < \omega < \omega_0$, then for
  any $(\mathcal{G}, \omega)$-representation $\pi$ we have
  $\pi_\sharp(q_0) - \pi_\sharp(q_1) = \pi_\sharp(q_0') -
  \pi_\sharp(q_1')$.  We are therefore free to prove the theorem for a
  convenient choice of idempotents.
  
  It is known that the fundamental class of $M$ in $K_0(M)$ coincides
  with $[\bar{\partial}_g] + (g-1)[\iota]$ where $\bar{\partial}_g$ is
  the Dolbeault operator on $M$ and $\iota\colon C(M)\to \C$ is a
  character (see \cite[Lemma~7.9]{Bettaieb-Matthey-etal05}).  Let
  $e_0, e_1, f_0, f_1$ be idempotents in some matrix algebra over
  $\ell^1(G)$ such that
  \begin{displaymath}
    \mu[\bar{\partial}_g] = [e_0] - [e_1]\quad
    \text{and}\quad
    \mu[\iota] = [f_0] - [f_1].
  \end{displaymath}
  (This gives an obvious choice of idempotents $q_0'$ and $q_1'$ in
  some matrix algebra over $\ell^1(G)$ so that $\mu[M] = [q_0'] -
  [q_1']$.)  We want to prove that
  \begin{displaymath}
    \tau\big( \pi_\sharp(\mu(z)) \big) =
    \langle \ch_\tau (\ell_\pi), \ch(z) \rangle
  \end{displaymath}
  for $z = [M] \in K_0(M)$.  Because of the additivity of this last
  equation, the fact that $[M] = [\bar{\partial}_g] + (g-1)[\iota]$,
  and Eq.~(\ref{eq:1}), it is enough to prove that
  \begin{equation}
    \label{eq:2}
    \tau\bigl( \pi_\sharp(\mu[\iota] )\bigr)
    =
    \langle \ch_\tau (\ell_\pi), \ch[\iota] \rangle.
  \end{equation}
  By \cite[Corollary 3.5]{Dadarlat12}
  \begin{equation*}
    \label{eq:3}
    \tau\big( \pi_\sharp(\mu[\iota]) \big)
    =
    \tau\big( \langle \ell_\pi, [\iota]\otimes 1_A \rangle \big).
  \end{equation*}
  We can represent $\ell_\pi$ by a projection $f$ in matrices over
  $C(M,A)$.  The definition of the Kasparov product implies that
  \begin{displaymath}
    \langle [\ell_\pi], [\iota]\otimes 1_A\rangle
    = \iota_*[f] = [f(x_0)] \in K_0(A).
  \end{displaymath}

  On the other hand, the definition of $\ch_\tau$ (see
  \cite[Definition~4.1]{Schick05}) implies that $\ch_\tau(f) =
  \tau(f(x_0))$ $+$ a term in $H^2(M, \R)$.  Since $\ch[\iota] = 1\in
  H_0(M, \R)$, we get
  \begin{equation}\label{eq:4}
    \langle \ch_\tau(f), \ch[\iota] \rangle
    = \tau(f(x_0)).
  \end{equation}
\end{proof}

\subsection{Statement of the main result}
We will often write $\Sigma_g$ for the closed oriented surface of
genus $g$ and $\Gamma_g$ for its fundamental group.  It is well known
that $\Gamma_g$ has a standard presentation
\begin{displaymath}
  \Gamma_g= \bigg\langle
  \alpha_1, \beta_1, \dots, \alpha_g, \beta_g \ \bigg|\
  \prod_{i=1}^g [\alpha_i,\beta_i]
  \bigg\rangle,
\end{displaymath}
where we write $[\alpha, \beta]$ for the multiplicative commutator
$\alpha\beta\alpha^{-1}\beta^{-1}$.

Our main result is the following.

\begin{theorem}
  \label{thm:Mg-formula}
  Let $g\geq 1$ be an integer and let $q_0$ and $q_1$ be idempotents
  in some matrix algebra over $\ell^1(\Gamma_g)$ such that
  $\mu[\Sigma_g] = [q_0] - [q_1]\in K_0(\ell^1(\Gamma_g))$.  There
  exists $\epsilon_0 > 0$ and a finite subset ${\cal F}_0$ of
  $\Gamma_g$ such that for every $0< \epsilon < \epsilon_0$ and every
  finite subset $\mathcal{F} \supseteq \mathcal{F}_0$ of $\Gamma_g$
  the following holds.

  If $A$ is a unital $C^*$-algebra with a trace $\tau$ and $\pi\colon
  \Gamma_g\to U(A)$ is an $(\mathcal{F}, \epsilon)$-representation,
  then
  \begin{equation}\label{eq:5}
    \tau \big( \pi_\sharp(\mu[\Sigma_g]) \big) 
    = \frac{1}{2\pi i} \tau\biggl(
    \log\bigg( \prod_{i=1}^g \big[\pi(\alpha_i),\pi(\beta_i)\big] \biggr)
    \bigg),
  \end{equation}
  where $\pi_\sharp(\mu[\Sigma_g]) := \pi_\sharp(q_0) -
  \pi_\sharp(q_1).$
\end{theorem}

The rest of the paper is devoted to the proof.

\begin{remark}
  \label{rem:exel-loring-orig}
  The case when $g=1$ and $A = M_n(\C)$ recovers the Exel-Loring
  formula of \cite{Exel-Loring91}.  As mentioned in the introduction,
  this formula states that two integer-valued invariants $\kappa(u,v)$
  and $\omega(u,v)$ associated to a pair unitary matrices $u ,v\in
  U(n)$ coincide as long as $\norm{ [u,v] - 1_n } < c$, where $c$ is a
  small constant that is independent $n$.  Exel observed (in
  \cite[Lemma~3.1]{Exel93}) that the ``winding-number'' invariant
  $\omega(u,v)$ equals
  \[
  \frac{1}{2\pi i} \tr\bigl( \log( [u,v] ) \bigr).
  \]
  This corresponds to the right-hand side of \eqref{eq:5}.
  
  The ``$K$-theory invariant'' $\kappa(u,v)$ is an element in the
  $K_0$-group of $M_n(\C)$ and is regarded as an integer after
  identifying this group with $\Z$ (the isomorphism given by the usual
  trace on $M_n(\C)$).  This invariant was first introduced
  by Loring in \cite{Loring88}.  We briefly recall its definition.

  Recall that $K_0( C(\T^2) ) \cong \Z[1] \oplus \Z [\beta]$ where
  $[1]$ is the class of the unit of $C(\T^2)$ and $\beta$ is the Bott
  element.  Given unitaries $U$ and $V$ in a unital $C^*$-algebra
  $B$, one defines a self-adjoint matrix
  \[
  e(U, V) =
  \begin{pmatrix}
    f(V)         & g(V) + h(V)U^*\\
    g(V) + Uh(V) & 1 - f(V)
  \end{pmatrix},
  \]
  where $f$, $g$, and $h$ are certain continuous functions on the
  circle.  These are chosen in such a way that when $U = e^{2\pi i x},
  V = e^{2 \pi i y}\in C(\T^2)$ we have that $e(U, V)$ is idempotent
  and has $K_0$-class $[1] + \beta$ (cf. \cite{Rieffel81, Loring88}).

  We assume that $\norm{ [u,v] - 1_n }$ is small enough so that the
  corresponding matrix $e(u,v)$ is nearly idempotent; in particular
  its spectrum does not contain $1/2$.  Writing $\chi$ for the
  characteristic function of $\{\operatorname{Re} z > 1/2\}$, we have
  that $\chi(e(u,v))$ is a projection in $M_n(\C)$.  Define
  \[
  \kappa(u,v) = \tr\big(\chi( e(u,v) )\big) - n.
  \]
  Subtracting $n$ means cancelling-out the $K$-theoretic contribution
  of $[1]$, leaving only the contribution of the push-forward of
  $\beta$ under a quasi-representation determined by $u$ and $v$.
  (Proposition~\ref{prop:existence-of-quasi} makes this statement
  precise.)  This corresponds to the left-hand side of \eqref{eq:5}.
  
  The formula \eqref{eq:5} also recovers its extension by H.~Lin
  \cite{Lin11} for $C^*$-algebras of tracial rank one.  Lin's strategy
  was a reduction to the finite-dimensional case of
  \cite{Exel-Loring91} using approximation techniques.
\end{remark}

The following proposition says that we may associate
quasi-representations with unitaries that nearly satisfy the group
relation in the standard presentation of $\Gamma_g$ mentioned above.
The proof is in Section~\ref{sec:proof-main-result}.

\begin{proposition}
  \label{prop:existence-of-quasi}
  For every $\epsilon > 0$ and every finite subset $\mathcal{F}$ of
  $\Gamma_g$ there is a $\delta > 0$ such that if $A$ is a unital
  $C^*$-algebra with a trace $\tau$ and $u_1,v_1,...,u_g,v_g$ are
  unitaries in $A$ satisfying
  \begin{displaymath}
    \biggl\| \prod_{i=1}^g [u_i,v_i] - 1 \biggr\| < \delta,
  \end{displaymath}
  then there exists an $({\cal F}, \epsilon)$-representation
  $\pi\colon \Gamma_g\to U(A)$ with $\pi(\alpha_i) = u_i$ and
  $\pi(\beta_i) = v_i$, for all $i\in \{1, \dots, g\}$.
\end{proposition}

\begin{example*}
  To revisit a classic example, consider the noncommutative 2-torus
  $A_\theta$, regarded as the universal $C^*$-algebra generated by
  unitaries $u$ and $v$ with $[v,u] = e^{2\pi i \theta}\cdot 1$.  This
  is a tracial unital $C^*$-algebra.  If $\theta$ is small enough, we
  may apply Proposition~\ref{prop:existence-of-quasi} and
  Theorem~\ref{thm:Mg-formula} to obtain
  \begin{displaymath}
    \tau(\pi_\sharp(\beta)) =
    \frac{1}{2\pi i} \tau( \log e^{-2 \pi i \theta} ) = -\theta
  \end{displaymath}
  where $\beta \in K_0(C(\T^2))$ is the Bott element, $\tau$ is a
  unital trace of $A_\theta$, and $\pi\colon \Z^2\to U(A_\theta)$ is a
  quasi-representation obtained from
  Proposition~\ref{prop:existence-of-quasi}.
\end{example*}

\section{The Mishchenko line bundle}
\label{sec:misch-line-bundle}

Recall our setup: $M$ is a closed oriented surface with fundamental
group $G$ and universal cover $p\colon \widetilde{M}\to M$.  In this
section we give a picture of the Mishchenko line bundle that will
enable us to explicitly describe its push-forward by a
quasi-representation.

The Mishchenko line bundle is the bundle $\widetilde{M}\times_G
\ell^1(G)\to M$, obtained from $\widetilde{M}\times \ell^1(G)$ by
passing to the quotient with respect to the diagonal action of $G$.
We write $\ell$ for its class in $K_0(C(M)\otimes \ell^1(G))$.

\subsection{Triangulations and the edge-path group}
\label{subsec:s_ij}
We adapt a construction found in the appendix of
\cite{Phillips-Stone90}.  It is convenient to work with a
triangulation $\Lambda$ of $M$.  Let $\Lambda^{(0)} = \{x_0, \dots,
x_{N-1}\}$ be the 0-skeleton of $\Lambda$ and $\Lambda^{(1)}$ be the
1-skeleton.  To each edge we assign an element of $G$ as follows.  Fix
a root vertex $x_0$ and a maximal (spanning) tree $T$ in $\Lambda$.
Let $\gamma_i$ be the unique path along $T$ from $x_{0}$ to $x_i$, and
for two adjacent vertices $x_i$ and $x_j$ let $x_ix_j$ be the
(directed) edge from $x_i$ to $x_j$.  For two such adjacent vertices,
write $s_{ij}\in G$ for the class of the loop $\gamma_i * x_ix_j *
\gamma_j^{-1}$.

Let $\mathcal{F}$ be the (finite) set $\{s_{ij}\}$.  For example, if
$M = \T^2$ so that $G = \Z^2 = \langle \alpha, \beta : [\alpha, \beta]
= 1\rangle$, we have ${\cal F} = \{1, \alpha^{\pm 1}, \beta^{\pm 1},
(\alpha\beta)^{\pm 1}\}$ for the triangulation and tree pictured in
Figure~\ref{fig:m1-triangulated} (on page
\pageref{fig:m1-triangulated})

\begin{definition}
  For a vertex $x_{i_k}$ in a 2-simplex $\sigma = \langle x_{i_0},
  x_{i_1}, x_{i_2}\rangle$ of $\Lambda$, define the \emph{dual cell
    block} to $x_{i_k}$
  \begin{displaymath}
    U_{i_k}^\sigma :=
    \bigg\{
      \sum_{l = 0}^2 t_l x_{i_l} : t_l\geq 0,
      \sum_{l = 0}^2 t_l = 1,
      \text{ and }
      t_{i_k}\geq t_l\text{ for all } l
    \bigg\}.
  \end{displaymath}
  Define the \emph{dual cell} to the vertex $x_i\in \Lambda^{(0)}$ by
  \begin{displaymath}
    U_i = \cup \{ U_i^\sigma : x_i \in \sigma\}.
  \end{displaymath}
  Let $U_{ij}^\sigma = U_i^\sigma\cap U_j^\sigma$ etc.  (See
  Figure~\ref{fig:dual-cells}.)
\end{definition}

\begin{figure}
  \centering
  \subfloat[]{\includegraphics{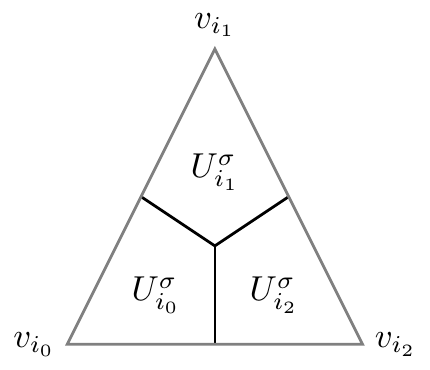}}\hfill
  \subfloat[]{\includegraphics{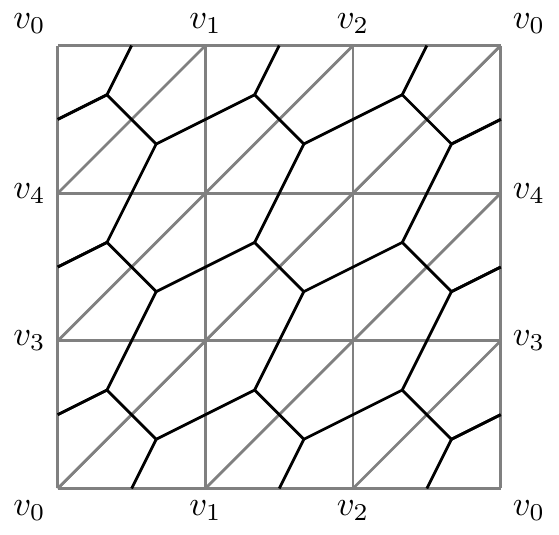}}
  \caption{%
    (\textsc{a}) Dual cell blocks in a simplex $\sigma =
    \langle v_{i_0}, v_{i_1}, v_{i_2}\rangle$.
    (\textsc{b}) A triangulation of $\T^2$ with the dual cell
    structure highlighted.}
  \label{fig:dual-cells}
\end{figure}

Since $p\colon \widetilde{M}\to M$ is a covering space of $M$, we may
fix an open cover of $M$ such that for every element $V$ of this
cover, $p^{-1}(V)$ is a disjoint union of open subsets of
$\widetilde{M}$, each of which is mapped homeomorphically onto $V$ by
$p$.  We require that $\Lambda$ be fine enough so that every dual cell
$U_i$ is contained in some element of this cover.

\begin{lemma}
  \label{lem:mishchenko-line-bundle}
  The Mishchenko line bundle $\widetilde{M}\times_G \ell^1(G)\to M$ is
  isomorphic to the bundle $E$ obtained from the disjoint union
  $\bigsqcup U_i \times \ell^1(G)$ by identifying $(x, a)$ with $(x,
  s_{ij}a)$ whenever $x\in U_i\cap U_j$.
\end{lemma}

\begin{proof}
  Lift $x_0$ to a vertex $\widetilde{x}_0$ in $\widetilde{M}$.  By the
  unique path-lifting property, every path $\gamma_i$ lifts (uniquely)
  to a path $\widetilde{\gamma}_i$ from $\widetilde{x}_{0}$ to a lift
  $\widetilde{x}_i$ of $x_i$.  In this way lift $T$ to a tree
  $\widetilde{T}$ in $\widetilde{M}$.  Each $U_i$ also lifts to a dual
  cell to $\widetilde{x}_i$, denoted $\widetilde{U}_i$, which $p$ maps
  homeomorphically onto $U_i$.

  We first describe the cocycle (transition functions) for the
  Mishchenko line bundle.  Identify the fundamental group $G$ of $M$
  with the group of deck transformations of $\widetilde{M}$; see for
  example \cite[Proposition~1.39]{Hatcher02}.  Use this to write
  $p^{-1}(U_i)$ as the disjoint union $\sqcup \{ s\widetilde{U}_i :
  s\in G\}$.  Consider the isomorphism $\Phi_i\colon
  p^{-1}(U_i)\times_G \ell^1(G) \to U_i\times \ell^1(G)$ described by
  the following diagram:
  \begin{displaymath}
    \includegraphics{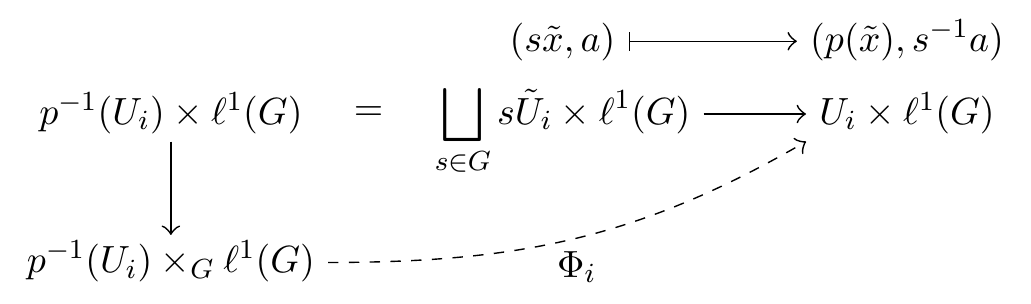}
  \end{displaymath}
  If $U_{ij} := U_i\cap U_j\not= \emptyset$, we obtain the cocycle
  $\phi_{ij}\colon U_{ij} \to \Aut(\ell^1(G))$:
  \begin{gather*}
    U_{ij}\times \ell^1(G) \overset{\Phi_j^{-1}}{\longrightarrow}
    p^{-1}(U_{ij})\times_G \ell^1(G) \overset{\Phi_i}{\longrightarrow}
    U_{ij}\times \ell^1(G)\\
    (x,a) \longmapsto (x,\phi_{ij}(x)a)
  \end{gather*}
  Observe that $\widetilde{M}\times_G \ell^1(G)$ is isomorphic to the
  bundle obtained from the disjoint union $\bigsqcup U_i\times
  \ell^1(G)$ by identifying $(x,a)$ with $(x, \phi_{ij}(x) a)$
  whenever $x\in U_{ij}$.  We only need to prove that $\phi_{ij}$ is
  constantly equal to $s_{ij}$.
  
  Let $x\in U_{ij}$ and let $\widetilde{x}\in \widetilde{U}_j$ be a
  lift of $x$.  Then $\Phi_j\bigl([\widetilde{x}, a]\bigr) = (x,a)$.
  Because $p(\widetilde{x})\in U_{ij}$ there is a (unique) $s\in G$
  such that $\widetilde{x}\in s\widetilde{U}_i\cap \widetilde{U}_{j}
  \not= \emptyset$.  Thus $\Phi_i\bigl([\widetilde{x}, a]\bigr) = (x,
  s^{-1}a)$.  Now, the path $s\widetilde{\gamma}_i *
  s\widetilde{x}_i\, \widetilde{x}_j * \widetilde{\gamma}_j^{-1}$,
  starts at $s\widetilde{x}_{0}$ and ends at $\widetilde{x}_{0}$.  Its
  projection in $M$ is the loop defining $s_{ij}$, so $s^{-1} =
  s_{ij}$.  Thus $\phi_{ij}(x) = s_{ij}$.
\end{proof}

\subsection{The push-forward of the line bundle}
\label{subsec:push-forward-line}
We will need an open cover of $M$, so we dilate the dual cells $U_i$
to obtain one.  Let $0 < \delta < 1/2$ and define $V^\sigma_i$ to be
the $\delta$-neighborhood of $U^\sigma_i$ intersected with $\sigma$.
As before, set $V_i = \bigcup_\sigma V_i^\sigma$.  Let $\{\chi_i\}$ be
a partition of unity subordinate to $\{V_i\}$.

By Lemma~\ref{lem:mishchenko-line-bundle} the class of the Mishchenko
line bundle in $K_0(C(M)\otimes \ell^1(G))$, denoted earlier by
$\ell$, corresponds to the class of the projection
\begin{displaymath}
  e := \sum_{i,j} e_{ij}\otimes \chi_i^{1/2}\chi_j^{1/2}\otimes s_{ij}
  \in M_N(\C)\otimes C(M)\otimes \ell^1(G),
\end{displaymath}
where $\{e_{ij}\}$ are the canonical matrix units of $M_N(\C)$ and $N$
is the number of vertices in $\Lambda$.

We may fix a pair of idempotents $q_0$ and $q_1$ in some matrix
algebra over $\ell^1(G)$ satisfying $[q_0] - [q_1] = \mu[M]\in
K_0(\ell^1(G))$.  Let $\omega > 0$ be given by
Theorem~\ref{thm:index-theorem}.  (We may assume that $\omega < 1/4$.)

Fix $0 < \epsilon < \omega$ and an $(\mathcal{F},
\epsilon)$-representation $\pi\colon G\to U(A)$.  We recall the
following notation from the introduction.

\begin{notation}
  For an $(\mathcal{F}, \epsilon)$-representation $\pi\colon G\to U(A)$
  as above, let
  \begin{displaymath}
      \ell_\pi := (\id_{C(M)}\otimes \pi)_\sharp(e).
  \end{displaymath}
\end{notation}

\section{Hilbert-module bundles and quasi-representations}
\label{sec:hilb-module-bundles}

As mentioned in the introduction, in \cite{Connes-Gromov-etal90} a
quasi-representation (with scalar values) of the fundamental group of
a manifold is associated to an ``almost-flat'' bundle over the
manifold.  In this section we instead define a canonical bundle
$E_\pi$ over $M$ associated with quasi-representation $\pi$.  Its
class in $K_0(C(M)\otimes A)$ will be the class $\ell_\pi$ of the
push-forward of the Mishchenko line bundle by $\pi$.  Our construction
will be explicit enough so that we can use Chern-Weil theory for such
bundles to analyze $\ch_\tau(\ell_\pi)$, see \cite{Schick05}.

Recall that $A$ is a $C^*$-algebra with trace $\tau$.

\begin{definition*}
  Let $X$ be a locally compact Hausdorff space.  A \emph{Hilbert
    $A$-module bundle} $W$ over $X$ is a topological space $W$ with a
  projection $W\to X$ such that the fiber over each point has the
  structure of a Hilbert $A$-module $V$, and with local
  trivializations $W|_U \overset{\sim}{\longrightarrow} U\times V$
  which are fiberwise Hilbert $A$-module isomorphisms.
\end{definition*}

We should point out that the $K_0$-group of the $C^*$-algebra
$C(M)\otimes A$ is isomorphic to the Grothendieck group of isomorphism
classes of finitely generated projective Hilbert $A$-module bundles
over $M$.  We identify the two groups.

\subsection{Constructing bundles}
\label{subsec:constr-almost-flat}

We adapt a construction found in \cite{Phillips-Stone90}.

First we define a family of maps $\{u_{ij}\colon U_{ij} \to \GL(A)\}$
satisfying
\begin{align*}
  u_{ji}(x) &= u_{ij}^{-1}(x),\qquad x\in U_{ij},\\
  u_{ik}(x) &= u_{ij}(x) u_{jk}(x),\qquad x\in U_{ijk}.
\end{align*}
These maps will be then extended to a cocycle defined on the
collection $\{V_{ij}\}$.

Following \cite{Phillips-Stone90} we will find it convenient to fix a
partial order \textbf{o} on the vertices of $\Lambda$ such that the
vertices of each simplex form a totally ordered subset.  We then call
$\Lambda$ a \emph{locally ordered} simplicial complex.  One may always
assume such an order exists by passing to the first barycentric
subdivision of $\Lambda$: if $\hat{\sigma}_1$ and $\hat{\sigma}_2$ are
the barycenters of simplices $\sigma_1$ and $\sigma_2$ of $\Lambda$,
define $\hat{\sigma}_1 < \hat{\sigma}_2$ if $\sigma_1$ is a face of
$\sigma_2$.

Consider a simplex $\sigma = \langle x_{i_0}, x_{i_1}, x_{i_2}
\rangle$ (with vertices written in increasing \textbf{o}-order).
Observe that in this case $U_{i_0}^\sigma\cap U_{i_2}^\sigma = U_{i_0
  i_2}^\sigma$ may be described using a single parameter $t_1$:
\begin{displaymath}
  U_{i_0 i_2}^\sigma = \biggl\{ \sum_{l=0}^2 t_l
  x_{i_l} : t_0 = t_2 = \frac{1 - t_1}{2} : 0\leq t_1\leq 1/3
  \biggr\}.
\end{displaymath}
Define
\begin{align*}
  u_{i_0 i_1}^\sigma &= \text{ the constant function on }
  U_{i_0 i_1}^\sigma \text{ equal to } \pi(s_{i_0 i_1})\\
  u_{i_1 i_2}^\sigma &= \text{ the constant function on }
  U_{i_1 i_2}^\sigma \text{ equal to } \pi(s_{i_1 i_2})\\
  u_{i_0 i_2}^\sigma (t_1) &= (1-3t_1)\pi(s_{i_0 i_2}) + 3t_1 \pi(s_{i_0 i_1})
  \pi(s_{i_1 i_2}),\qquad 0\leq t_1\leq 1/3.
\end{align*}
Define $u_{i_2 i_0}^\sigma$ etc. to be the pointwise inverse of
$u_{i_0 i_2}^\sigma$.  For fixed $i$ and $j$, the maps
$u_{ij}^\sigma\colon U_{ij}^\sigma\to \GL(A)$ define a map
$u_{ij}\colon U_{ij}\to \GL(A)$.  Indeed, if $x_ix_j$ is a common edge
of two simplices $\sigma$ and $\sigma'$, then $U_{ij}^\sigma \cap
U_{ij}^{\sigma'}$ is the barycenter of $\langle x_i, x_j \rangle$,
where by definition both $u_{ij}^\sigma$ and $u_{ij}^{\sigma'}$ take
the value $\pi(s_{ij})$. By construction the family $\{u_{ij}\}$ has
the desired properties.

\subsubsection{}
\label{subsubsec:v-ij}
Recall the sets $V_i$ etc.\ from
Section~\ref{subsec:push-forward-line}.  To define the smooth
transition function $v^\sigma_{i_0 i_2}\colon V^\sigma_{i_0 i_2}\to
\GL(A)$ that will replace $u_{i_0 i_2}^\sigma$, let us assume for
simplicity that the simplex $\sigma$ is the triangle with vertices
$v_{i_0}=(-1/2,0)$, $v_{i_1} = (0,1)$, and $v_{i_2}=(1/2,0)$.  (It may
be helpful to consider Figure~\ref{fig:dual-cells}\textsc{a}.)

Define $v^\sigma_{i_0i_2}$ as follows:
\begin{displaymath}
  v^\sigma_{i_0 i_2}(x,y) =
  \begin{cases}
    \pi(s_{i_0 i_1})\pi(s_{i_1 i_2}),
     &1/3 - \delta \leq y \leq 1/3 + \delta\\
    (1 - \frac{y}{1/3 - \delta})\pi(s_{i_0 i_2}) +\\
    \quad + \frac{y}{1/3 - \delta}\pi(s_{i_0 i_1})\pi(s_{i_1 i_2}),
     & 0 \leq y \leq 1/3 - \delta
  \end{cases}
\end{displaymath}
(so $v^\sigma_{i_0i_2}$ is constant along the horizontal segments in
$V_{i_0 i_2}$).  The remaining two transition functions remain
constant:
\begin{align*}
  v_{i_0 i_1}^\sigma &= \pi(s_{i_0 i_1})\\
  v_{i_1 i_2}^\sigma &= \pi(s_{i_1 i_2})
\end{align*}
Again, for fixed $i$ and $j$ the maps $v_{ij}^\sigma\colon
V_{ij}^\sigma\to \GL(A)$ define a map $v_{ij}\colon V_{ij}\to \GL(A)$.
Since $v^\sigma_{i_0 i_2}$ is constant and equal to $\pi(s_{i_0
  i_1})\pi(s_{i_1 i_2})$ in $V_{i_0}\cap V_{i_2} \cap V_{i_2}$, we
indeed obtain a family $\{v_{ij}\}$ of transition functions.

\begin{definition}
  The Hilbert $A$-module bundle $E_\pi$ is constructed from the
  disjoint union $\bigsqcup V_{i} \times A$ by identifying $(x, a)$
  with $(x, v_{ij}(x)a)$ for $x$ in $V_{ij}$.
\end{definition}

\begin{proposition}
  \label{prop:E_pi-and-l_pi}
  The class of $E_\pi$ in $K_0(C(M)\otimes A)$ coincides with
  $\ell_\pi$, the class of the push-forward of $e$ by
  $\id_{C(M)}\otimes \pi$ (see
  Section~\ref{subsec:push-forward-line}).
\end{proposition}

\begin{proof}
  The bundle $E_\pi$ is a quotient of $\bigsqcup V_i \times A$ and
  from its definition it is clear that for each $i$ the quotient map
  is injective on $V_i\times A$.  The restriction of the quotient map
  to $V_i\times A$ has an inverse, call it $\psi_i$, and $\psi_i$ is a
  trivialization of $E_\pi|_{V_i}$.  Recalling that $N$ is the number
  of vertices in $\Lambda$ (which is the same as the number of sets
  $V_i$ in the cover), we define an isometric embedding
  \begin{gather*}
    \theta\colon E_\pi \to M\times A^N\\
    [x,a] \mapsto \bigl( \chi_i^{1/2}(x)\psi_i([x,a])
    \bigr)_{i=0}^{N-1}.
  \end{gather*}
  Let $e_\pi\colon M\to M_N(A)$ be the function
  \begin{displaymath}
    x\mapsto \sum_{i,j} e_{ij}\otimes \chi_i^{1/2}(x)\chi_j^{1/2}(x)
    v_{ij}(x).
  \end{displaymath}
  Because $\psi_i\psi_j^{-1}(x,a) = (x, v_{ij}(x)a)$ for $x\in
  V_{ij}$, it is easy to check that $e_\pi(x)$ is the matrix
  representing the orthogonal projection of $A^N$ onto
  $\theta(E_\pi|_x)$.  In this way we see that $[E_\pi] = [e_\pi]
  \in K_0(C(M)\otimes A)$.

  Since $\mathcal{F} = \{s_{ij}\}$ and $\pi$ is an $(\mathcal{F},
  \epsilon)$-representation, it follows immediately that the
  transition functions $v_{ij}$ satisfy $\norm{v_{ij}(x) -
    \pi(s_{ij})} < \epsilon$ for all $x\in V_{ij}$.  Thus
  \begin{displaymath}
    \norm{e_\pi - (1\otimes \pi)(e)} = \biggl\| e_\pi -
    \sum_{i,j} e_{ij}\otimes \chi_i^{1/2}\chi_j^{1/2} \pi(s_{ij})
    \biggr\| < \epsilon
  \end{displaymath}
  as well.  Recall that $\ell_\pi$ is obtained by perturbing
  $(1\otimes \pi)(e)$ to a projection using functional calculus and
  then taking its $K_0$-class (see
  Section~\ref{subsec:pushing-forward-via}).  The previous estimate
  shows that this class must be $[e_\pi]$.
\end{proof}

\begin{remark}
  The previous proposition shows that the class $[E_\pi]$ is
  independent of the order \textbf{o} on the vertices of $\Lambda_0$.
\end{remark}

\subsection{Connections arising from transition functions}
\label{subsec:conn-aris-from}
We now define a canonical connection on $E_\pi$ associated with the
family $\{v_{ij}\}$ of transition functions.  This connection will be
used in the proof of Theorem~\ref{thm:chern-class}.

\subsubsection{}
\label{subsubsec:onnection-forms-def}
The smooth sections $\Gamma(E_\pi)$ of $E_\pi$ may be identified with
\begin{displaymath}
  \bigl\{ (s_i) \in \bigoplus_i \Omega^0(V_i, E_\pi) : s_j =
  v_{ji}s_i\text{ on }V_{ij} \bigr\}.
\end{displaymath}
Let $\del_i\colon \Omega^0(V_i, A) \to \Omega^1(V_i, A)$
be given by
\begin{displaymath}
  \del_i(s) = d s +
  \omega_i s\qquad \forall s\in \Omega^0(V_i, A),
\end{displaymath}
where
\begin{displaymath}
  \omega_i = \sum_k \chi_k v_{ki}^{-1} d
  v_{ki}.
\end{displaymath}
Notice that $v_{ki}\in \Omega^0(V_{ik}, \GL(A))$ and so $\omega_i$ may
be regarded as an $A$-valued 1-form on $V_i$, which can be multiplied
fiberwise by the values of the section $s$.

We define a connection $\del$ on $E_\pi$ by
\begin{displaymath}
  \del (s_i) = (\del_i s_i).
\end{displaymath}
That $\del$ takes values in $\Omega^1(M,E_\pi)$ follows from a
straightforward computation verifying
\begin{displaymath}
  \del_j s_j = v_{ji} \del_i s_i.
\end{displaymath}
It is just as straightforward to verify that $\del$ is $A$-linear and
satisfies the Leibniz rule.

\subsubsection{}
Define $\Omega_i = d\omega_i + \omega_i \wedge \omega_i \in
\Omega^2(V_i, A)$.  One checks that $\Omega_i = v_{ji}^{-1} \Omega_j
v_{ji}$ and so $(\Omega_i)$ defines an element $\Omega$ of
$\Omega^2(M, \End_A(E_\pi))$.  This is nothing but the curvature of
$\del$ (see \cite[Proposition~3.8]{Schick05}).

\section{The Chern character}
\label{sec:chern-class}

In this section we prove our main technical result,
Theorem~\ref{thm:Mg-formula}.  It computes the trace of the
push-forward of $\mu[M]$ in terms of the de la Harpe-Skandalis
determinant by using that the cocycle conditions almost hold for the
elements $\pi(s_{ij})$,

\subsection{The de la Harpe-Skandalis determinant}
The de la Harpe-Skandalis determinant \cite{Harpe-Skandalis84} appears
in our formula below.  Let us recall the definition.  Write
$\GL_\infty(A)$ for the (algebraic) inductive limit of
$(\GL_n(A))_{n\geq1}$ with standard inclusions.  For a piecewise
smooth path $\xi\colon [t_1, t_2] \to \GL_\infty(A)$, define
\begin{align*}
  \widetilde{\Delta}_\tau(\xi) &= \frac{1}{2\pi i} \, \tau
  \biggl(%
    \,\int_{t_1}^{t_2} \xi'(t) \xi(t)^{-1}
    d t
  \biggr)
  = \frac{1}{2\pi i} \int_{t_1}^{t_2} \tau ( \xi'(t)
  \xi(t)^{-1}) d t.
\end{align*}

We will make use of some of the properties of
$\widetilde{\Delta}_\tau$ stated below.

\begin{lemma}[cf. Lemme~1 of \cite{Harpe-Skandalis84}]
  \label{lem:de-la-harpe}
  \mbox{}
  \begin{enumerate}
  \item Let $\xi_1, \xi_2\colon [t_1, t_2]\to \GL^0_\infty(A)$ be two
    paths and $\xi$ be their pointwise product.  Then
    $\widetilde{\Delta}_\tau(\xi) = \widetilde{\Delta}_\tau(\xi_1) +
    \widetilde{\Delta}_\tau(\xi_2)$.
  \item Let $\xi\colon [t_1, t_2]\to \GL_\infty^0(A)$ be a path with
    $\norm{\xi(t) - 1} < 1$ for all $t$.  Then
    \begin{displaymath}
      2\pi i \cdot\widetilde{\Delta}_\tau(\xi) =
      \tau \bigl(\log \xi(t_2)\bigr) - \tau
      \bigl(\log\xi(t_1)\bigr).
    \end{displaymath}
  \item The integral $\widetilde{\Delta}_\tau(\xi)$ is left invariant
    under a fixed-end-point homotopy of $\xi$.
  \end{enumerate}
\end{lemma}

\subsection{The Chern character on $K_0(C(M)\otimes A)$}
Assume $\tau$ is a trace on $A$.  Then $\tau$ induces a map on
$\Omega^2(V_i, \End_A(E_\pi|_{V_i}))$ and by the trace property
$\tau(\Omega_i) = \tau(\Omega_j)$ on $V_{ij}$.  We obtain in this way
a globally defined form $\tau(\Omega) \in \Omega^2(M, \C)$.

Since the fibers of our bundle are all equal to $A$, and our manifold
is 2-dimensional, the definition of the Chern character associated
with $\tau$ (from \cite[Definition~4.1]{Schick05}, but we have
included a normalization coefficient) reduces to
\begin{multline}
  \label{eq:6}
  \ch_\tau(\ell_\pi) = \tau\left(\exp\left(\frac{i\Omega}{2\pi}\right)\right) =
  \tau\bigg(
  \sum_{k=0}^\infty \frac{i\Omega/2\pi\wedge \dots \wedge i\Omega/2\pi}{k!}
  \bigg) =\\
  = \tau\left(\frac{i\Omega}{2\pi}\right) \in \Omega^2(M, \C).
\end{multline}
This is a closed form whose cohomology class does not depend on the
choice of the connection $\del$ (see \cite[Lemma~4.2]{Schick05}).

\vspace{\baselineskip}
A few remarks are in order before stating the next result.

Because $\Lambda$ is a locally ordered simplicial complex (recall the
partial order \textbf{o} from Section~\ref{subsec:constr-almost-flat}), every
2-simplex $\sigma$ may be written uniquely as $\langle x_i, x_j,
x_k\rangle$ with the vertices written in increasing \textbf{o}-order.
Whenever we write a simplex in this way it is implicit that the
vertices are written in increasing \textbf{o}-order.  We may write
$\sigma$ for $\sigma$ along with this order.

The orientation $[M]$ induces an orientation of the boundary of the
dual cell $U_i$ and in particular of the segment $U_{ik}^\sigma$.  Let
$s(\sigma) = 0$ if the initial endpoint of $U_{ik}^\sigma$ under this
orientation is the barycenter of $\sigma$, and let $s(\sigma) = 1$
otherwise.

\begin{theorem}
  \label{thm:chern-class}
  For a simplex $\sigma = \langle x_i, x_j, x_k\rangle$ of $\Lambda$,
  let $\xi_\sigma$ be the linear path
  \begin{displaymath}
    \xi_\sigma(t) = (1-t)\pi(s_{ik}) + t\pi(s_{ij})\pi(s_{jk}),\quad t\in[0,1]
  \end{displaymath}
  in $\GL(A)$.  Then
  \begin{displaymath}
    \tau \bigl( \pi_\sharp(\mu[M]) \bigr) =
    \sum_{\sigma}
    (-1)^{s(\sigma)} \widetilde{\Delta}_\tau(\xi_{\sigma}),
  \end{displaymath}
  where the sum ranges over all 2-simplices $\sigma$ of $\Lambda$.
\end{theorem}

\begin{proof}
  The path $\xi_{\sigma}$ lies entirely in $\GL(A)$ because
  $\norm{\pi(s_{ik}) - \pi(s_{ij})\pi(s_{jk})} < \epsilon$.  It
  follows from Theorem~\ref{thm:index-theorem} (on
  page~\pageref{thm:index-theorem}) and Eq.~(\ref{eq:6}) above
  that
  \begin{displaymath}
    \tau \bigl( \pi_\sharp(\mu[M]) \bigr)
    = \langle \ch_\tau(\ell_\pi), [M] \rangle
    =  -\frac{1}{2\pi i} \int_M \tau(\Omega).
  \end{displaymath}
  We compute this integral.

  First observe that by the trace property of $\tau$ we have
  $\tau(\omega_l \wedge \omega_l) = 0$ for every $l$.  Thus
  \begin{multline*}
    \int_M \tau(\Omega)
    = \sum_l \int_{U_l} \tau(\Omega_l)
    = \sum_l \int_{U_l}
       \tau( d\omega_l + \omega_l\wedge \omega_l ) =\\
    = \sum_l \int_{U_l} \tau( d\omega_l )
    = \sum_l \int_{U_l} d\tau( \omega_l )
    = \sum_l \int_{\partial U_l} \tau( \omega_l ),
  \end{multline*}
  where we used Green's theorem for the last equality and $\partial
  U_l$ has the orientation induced from $[M]$.  Recall that $U_l$ is
  the dual cell to $v_l$.  Write this as a sum over the 2-simplices of
  $\Lambda$:
  \begin{displaymath}
    \sum_l \int_{\partial U_l} \tau( \omega_l )
    = \sum_l \sum_{\sigma} \int_{(\partial U_l)\cap\sigma} \tau( \omega_l )
    = \sum_{\sigma} \sum_l \int_{(\partial U_l)\cap\sigma} \tau( \omega_l ).
  \end{displaymath}
  Exactly three dual cells meet a 2-simplex $\sigma = \langle x_{i},
  x_{j}, x_{k}\rangle$---$U_{i}$, $U_{j}$, and $U_{k}$---so for each
  simplex there are three integrals we need to account for.  Let us
  treat each of these in turn.

  The definition of the connection forms (see
  Section~\ref{subsubsec:onnection-forms-def}) implies that $\omega_{i}$
  restricted to $\sigma$ equals
  \begin{displaymath}
    \omega_{i} =
    \chi_{k} v_{k i}^{-1} d v_{k i} +
    \chi_{j} v_{j i}^{-1} d v_{j i} =
    \chi_{k} v_{k i}^{-1} d v_{k i},
  \end{displaymath}
  where the last equality follows from the fact that $v_{j i}$ is
  constant.  Now, $(\partial U_{i})\cap\sigma$ is the union of the two
  segments $U_{i j}^\sigma$ and $U_{i k}^\sigma$.  Observe that
  $v_{ik}$ is constantly equal to $\pi( s_{ij} ) \pi( s_{jk} )$ on
  $V_i\cap V_j\cap V_k$ (see Section~\ref{subsubsec:v-ij}). Since $U_{i
    j}^\sigma \cap V_{k}\subseteq V_i\cap V_j\cap V_k$ and $\chi_{k}$
  vanishes outside $V_{k}$, we get
  \begin{displaymath}
    \int_{(\partial U_{i})\cap\sigma} \tau( \omega_{i} ) =
    \int_{U_{i j}^\sigma} \tau(\chi_{k} v_{k i}^{-1} d v_{k i}) +
    \int_{U_{i k}^\sigma} \tau(\chi_{k} v_{k i}^{-1} d v_{k i})
    =  \int_{U_{i k}^\sigma} \tau(\chi_{k} v_{k i}^{-1} d v_{k i}).
  \end{displaymath}
  
  The second integral $\int_{(\partial U_{j})\cap\sigma} \tau(
  \omega_{j} )$ vanishes.  This is because $v_{ij}$ and $v_{jk}$ are
  constant and so
  \begin{displaymath}
    \omega_j =
    \chi_i v_{ij}^{-1} d v_{ij} +
    \chi_k v_{kj}^{-1} d v_{kj} = 0.
  \end{displaymath}

  The third integral may be calculated just as the first, with the
  roles of $i$ and $k$ reversed.  We obtain
  \begin{displaymath}
    \int_{(\partial U_{k})\cap\sigma} \tau( \omega_{k} ) 
    =  \int_{U_{k i}^\sigma} \tau(\chi_{i} v_{i k}^{-1} d v_{i k}).
  \end{displaymath}  
  Combining the three integrals we get
  \begin{multline*}
    \sum_{\sigma} \sum_l \int_{(\partial U_l)\cap\sigma} \tau(\omega_l)
    = \sum_\sigma \bigg(
      \int_{U_{i k}^\sigma} \tau(\chi_{k} v_{k i}^{-1} d v_{k i}) +
      \int_{U_{k i}^\sigma} \tau(\chi_{i} v_{i k}^{-1} d v_{i k})
      \bigg)=\\
    = \sum_\sigma \int_{U_{i k}^\sigma} \tau(
      \chi_{k} v_{k i}^{-1} d v_{k i} - \chi_{i} v_{i k}^{-1} d
      v_{i k} ),
  \end{multline*}
  where the last equality is due to the opposite orientations of the
  segment $U_{i k}^\sigma$ in the preceding two integrals.

  It follows from $v_{ik} v_{ki} = 1$ that $d v_{ik}\, v_{ik}^{-1} +
  v_{ki}^{-1}d v_{ki} =0$.  Therefore, the last line in the equation
  above is equal to
  \begin{displaymath}
    \sum_{\sigma}
      \int_{U_{ik}^\sigma} \tau(
        \chi_k v_{ki}^{-1} d v_{ki} +
        \chi_i v_{ki}^{-1} d v_{ki}
        )
    =
    \sum_{\sigma}
      \int_{U_{ik}^\sigma} \tau( v_{ki}^{-1} d v_{ki} )
    =
    -\sum_{\sigma}
      \int_{U_{ik}^\sigma} \tau( v_{ik}^{-1} d v_{ik} ).
  \end{displaymath}

  To arrive at the conclusion of the theorem, consider the restriction
  of $v_{ik}$ to the segment $U_{i k}^\sigma$.  This is the segment
  between the barycenter of $\sigma$, where $v_{i k}$ takes the value
  $\pi(s_{ij})\pi(s_{jk})$, and the barycenter of $\langle x_i,
  x_k\rangle$, where $v_{ik}$ takes the value $\pi(s_{ik})$ (see
  Section~\ref{subsubsec:v-ij}).  Then
  
  \begin{displaymath}
    \int_{U_{ik}^\sigma} \tau(v_{ki}^{-1} d v_{ki})
    =
    (-1)^{s(\sigma)}2\pi i\cdot \widetilde{\Delta}_\tau(\xi_\sigma).
  \end{displaymath}
  This concludes the proof.
\end{proof}

\section{Oriented surfaces}
\label{sec:oriented-surfaces}

For the proof of Theorem~\ref{thm:Mg-formula}, we will use a
convenient triangulation $\Lambda_g$ of the orientable genus $g$
surface $\Sigma_g$ that we proceed to describe.  The covering space of
$\Sigma_g$ is the open disc and we may take as a fundamental domain a
regular $4g$-gon, call it $\widetilde{\Sigma}_g$, drawn in the
hyperbolic plane.

Figure~\ref{fig:m2-octagon} depicts a procedure to obtain
$\widetilde{\Sigma}_2$ by gluing together two copies of
$\widetilde{\Sigma}_1$.  (We will give a more explicit description of
$\widetilde{\Sigma}_g$ in a moment).  It also illustrates the labeling
we use for the (oriented) sides of $\widetilde{\Sigma}_1$ and
$\widetilde{\Sigma}_2$.  To obtain $\Sigma_1$, for example, we
identify the side $\mathsf{a}$ with $*\mathsf{a}$ and the side
$\mathsf{b}$ with $*\mathsf{b}$.  To obtain the double torus
$\Sigma_2$, we identify $\mathsf{a}_k$ with $*\mathsf{a}_k$ and
$\mathsf{b}_k$ with $*\mathsf{b}_k$ for $k\in\{1, 2\}$.

\begin{figure}
  \centering
   \subfloat[]{\includegraphics{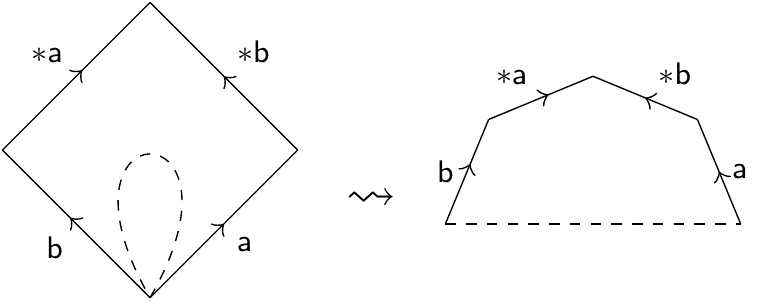}}\hfill
   \subfloat[]{\includegraphics{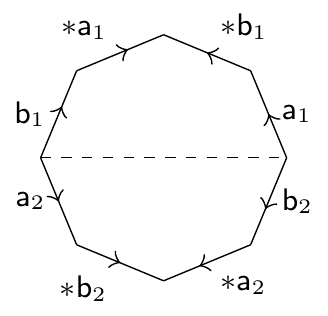}}\\
  \caption{
    (\textsc{a}) The fundamental domain $\widetilde{\Sigma}_1$ (with a hole).
    (\textsc{b}) The fundamental domain $\widetilde{\Sigma}_2$.
    }
  \label{fig:m2-octagon}
\end{figure}

\subsection{Triangulations}
Let us first define a triangulation $\widetilde{\Lambda}_g$ of the
fundamental domain $\widetilde{\Sigma}_g$. We do this by gluing $g$
triangulated copies of $\widetilde{\Sigma}_1$ together.
Figure~\ref{fig:k-copy}\textsc{a} on page~\pageref{fig:k-copy} shows
the triangulation for the $k$th copy of $\widetilde{\Sigma}_1$ (with a
hole), call it $\widetilde{\Lambda}_1^k$.  Ignore the labels on the
edges and the highlighted edges for now.  The vertex labeling also
indicates how to glue $\widetilde{\Lambda}_1^k$ to
$\widetilde{\Lambda}_1^{k-1}$ and $\widetilde{\Lambda}_1^{k+1}$, with
addition modulo $g$. Figure~\ref{fig:k-copy}\textsc{b} illustrates the
result of this gluing, the end-result being $\widetilde{\Lambda}_g$ by
definition.

The underlying space of $\widetilde{\Lambda}_g$ is
$\widetilde{\Sigma}_g$.  Identifying all the vertices $v_i^k$, as well
as identifying $a_i^k$ with $*a_i^k$ and $b_i^k$ with $*b_i^k$, for
each $i\in\{1,2\}$ and $k\in\{1, \dots, n\}$, yields a triangulation
$\Lambda_g$ of $\Sigma_g$.

\subsection{Surface groups}
\label{subsec:deck-trans}
We identify the fundamental group $\Gamma_g$ of $\Sigma_g$ with the
group of deck transformations of the universal covering space of
$\Sigma_g$.  We give a more concrete description of this group now.

The fundamental domain $\widetilde{\Sigma}_g$ is a regular $4g$-gon.
We write $\mathsf{a}_k$, $\mathsf{b}_k$ $*\mathsf{a}_k$ and
$*\mathsf{b}_k$, $k\in \{1,\dots, n\}$, for its (oriented) sides.  The
triangulation $\widetilde{\Lambda}_g$ gives a subdivision of the side
$\mathsf{a}_k$ into the three edges in the path $(v_0^k, a_1^k, a_2^k,
v_1^k)$ (with orientation given by the directed edge $(a_1^k,
a_2^k)$).  The subdivision of the sides $\mathsf{b}_k$,
$*\mathsf{a}_k$ and $*\mathsf{b}_k$ is similar.  See
Figure~\ref{fig:k-copy}\textsc{a}.

The group of deck transformations $\Gamma_g$ is generated by the
hyperbolic isometries $\alpha_k$ and $\beta_k$, $k\in \{1, \dots,
g\}$, defined as follows: $\alpha_k$ maps $*\mathsf{a}_k$ to
$\mathsf{a}_k$ in such a way that, locally, the half-plane bounded by
$*\mathsf{a}_k$ containing $\widetilde{\Sigma}_g$ is mapped to the
half-plane bounded by $\mathsf{a}_k$ but opposite
$\widetilde{\Sigma}_g$.  The transformation $\beta_k$ is defined
analogously, mapping $*\mathsf{b}_k$ to $\mathsf{b}_k$.  We refer the
reader to \cite[Chapter~VII]{Iversen92} for more details.  When $g=1$,
for example, the transformations $\alpha_1$ and $\beta_1$ are just
translations.  See Figure~\ref{fig:m1-triangulated}, where we have
omitted the sub- and superscripts corresponding to $k=1$, since $g=1$.

For $k\in \{1, \dots, g\}$, let
\begin{displaymath}
  \kappa_k = \prod_{j=1}^k [\alpha_k, \beta_k]
\end{displaymath}
and let $\kappa_0 = 1$.  We have that $\kappa_g = 1$.

\subsection{Local orders and trees}
We need $\Lambda_g$ to be locally ordered, so we proceed to fix a
partial order on the vertices of $\Lambda_g$ such that the vertices of
every simplex form a totally ordered set.  Let us define an order on
the vertices of $\widetilde{\Lambda}_g$ that drops down to the order
we need.  On the $k$th copy $\widetilde{\Lambda}_1^k$, the
corresponding order is indicated in Figure~\ref{fig:k-copy}\textsc{a}
by arrows on the edges, always pointing from a smaller vertex to a
larger one. It is defined as follows:
\begin{itemize}
\item for the ``inner'' vertices we go ``counter-clockwise'': for
  fixed $k\in\{1, \dots, g\}$, $w^k_i < w^k_j$ if $i < j$, except when
  $k=g$ and $j=4$ (in which case $w_4^g = w_0^1$ and we already have
  $w_0^1 < w_k^i$);
\item the ``inner'' vertices are larger than the ``outer'' ones:
  $w_i^k > v_j^l$, $a_j^l$, $b_j^l$, $*a_j^l$, $*b_j^l$ for all $i$,
  $j$, $k$ and $l$;
\item for the ``outer'' vertices: $v_i^k < a_j^l, b_j^l, *a_j^l,
  *b_j^l$ for all $i$, $j$, $k$ and $l$; for every $k$, $a_1^k <
  a_2^k$, $*a_1^k < *a_2^k$, and similarly for the $b_j^k$.
\end{itemize}

Finally, we will need a spanning tree $T_g$ of $\Lambda_g$, and a lift
$\widetilde{T}_g$ to the triangulation $\widetilde{\Lambda}_g$ of the
fundamental domain $\widetilde{\Sigma}_g$.  Again, we define
$\widetilde{T}_g$ first.  It is obtained as the union of the edge
between $w_0^1$ and $v_0^1$ (including those two vertices) and trees
in each copy $\Sigma_1^k$. The tree in $\Sigma_1^k$ is depicted in
Figure~\ref{fig:k-copy}\textsc{a} by highlighted (heavier) edges.
This drops to a spanning tree $T_g$ of $\Sigma_g$.  We regard $T_g$ as
``rooted'' at the vertex $v_0^1$.  (In the notation of
\ref{subsec:s_ij}, where the vertices were labeled consecutively as
$x_0, \dots, x_N$, we have that $v_0^1 = x_0$.)

\section{Proof of the main result}
\label{sec:proof-main-result}

This section contains the proof of Theorem~\ref{thm:Mg-formula}.  The
proof is split into a number of lemmas.

To apply Theorem~\ref{thm:chern-class} we will first compute the group
element $s_{ij}$ corresponding to each edge $x_ix_j$ of $\Lambda_g$ ,
in the sense discussed in Section~\ref{subsec:s_ij}.  Equivalently, we
compute group elements corresponding to edges in the cover
$\widetilde{\Lambda}_g$, keeping in mind that the lifts of any edge of
$\Lambda_g$ will all correspond to the same group element.

A concise way of stating the result of these computations is to label
each edge in Figure~\ref{fig:k-copy}\textsc{a} with the corresponding
group element.

\begin{lemma}
  \label{lem:edge-labels}
  The labels in Figure~\ref{fig:k-copy}\textsc{a} are correct.
\end{lemma}

\begin{proof}
  We carry out the computations in three separate claims.
  
  \begin{claim*}
    An edge of the form $a_i^k w_j^k$ corresponds to $\alpha_k^{-1}\in
    \Gamma_g$.  Similarly, an edge of the form $b_i^k w_j^k$
    corresponds to $\beta_k^{-1}\in \Gamma_g$.
  \end{claim*}
  Consider $a_i^k w_j^k$ first.  When we add this edge to the forest
  that is the union of all the lifts of $T_g$ (that is, translates of
  $\widetilde{T}_g$), we obtain a unique path $P$ between $v_0^1$, our
  root vertex, and some translate $s v_0^1$, where $s\in \Gamma_g$.
  We regard $P$ as directed in the direction of the edge $a_i^k w_j^k$
  that we started with, so it is a path from $s v_0^1$ to $v_0^1$.  It
  therefore drops down to a loop in $\Sigma_g$ whose class is
  $s^{-1}$, the group element we want to compute (see
  \cite[Proposition~1.39]{Hatcher02}, for example).  Now notice that
  because $*a_i^k$ belongs to $\widetilde{T}_g$, its translate
  $\alpha_k (*a_i^k) = a_i^k$ belongs to the translate $\alpha_k
  \widetilde{T}_g$ of $\widetilde{T}_g$.  Thus $P$ is a path between
  $v_0^k$ and $\alpha_k v_0^k$.  The corresponding group element is
  therefore $\alpha_k^{-1}$.  An entirely similar argument applies to
  the edge $b_i^k w_j^k$.
  
  \begin{claim*}
    Any edge between inner vertices (vertices of the form $w_i^k$)
    corresponds to $1\in \Gamma_g$.  The edges $a_1^k a_2^k$, $b_1^k
    b_2^k$, $*a_1^k {*a_2^k}$, and $*b_1^k {*b_2^k}$ all correspond to
    $1\in \Gamma_g$.
  \end{claim*}
  We proceed as in the previous claim.  Any edge between inner
  vertices is either in $\widetilde{T}_g$ or between two vertices that
  are in $\widetilde{T}_g$.  The associated path we get is therefore
  from $v_0^k$ to itself.  The same is true of the edges $b_1^k b_2^k$
  and $*a_1^k {*a_2^k}$.  It follows that the corresponding group
  element is $1$.  Since $a_1^k a_2^k$ and $*a_1^k {*a_2^k}$ are both
  lifts of the same edge, they correspond to the same element.
  Similarly, $*b_1^k {*b_2^k}$ corresponds to $1$.

  \begin{claim*}
    An edge that is incident to $v_i^k$ and to a vertex $z$ in the
    tree $\widetilde{T}_g$ corresponds to the element $s\in \Gamma_g$
    such that $v_0^1 = s v_i^k$.  (The edge is given the orientation
    induced by the order on the vertices, as usual.)  For $k\in \{1,
    \dots, g\}$,
    \begin{align*}
      v_0^1 &= \kappa_{k-1}\cdot v_0^k\\
      v_0^1 &= \kappa_{k-1} \alpha_k\beta_k\alpha_k^{-1}\cdot v_1^k\\
      v_0^1 &= \kappa_{k-1} \alpha_k\beta_k\cdot v_2^k\\
      v_0^1 &= \kappa_{k-1} \alpha_k\cdot v_3^k\\
    \end{align*}
    (Recall that
    $\kappa_k$ is the product of commutators
    $
      [\alpha_1, \beta_1] [\alpha_2, \beta_2] \cdots
      [\alpha_k, \beta_k]
    $
    for $k\in \{1, \dots, g\}$, and that $\kappa_0 = 1$.)
  \end{claim*}
  Observe that, because of how the order was defined, $v_i^k < z$
  always holds.  When we add the edge $v_i^k z$ to the tree
  $\widetilde{T}_g$ we obtain a path from $v_i^k$ to $v_0^1$.  (See
  Figure~\ref{fig:k-copy}, but keep in mind that in the case $k=1$ the
  edge $v_0^1 w_0^1$ belongs to the tree.)  It follows that the
  corresponding element is the $s\in \Gamma_g$ such that $v_0^1 = s
  v_i^k$.

  To compute these elements $s$ we argue by induction on $k$.  Assume
  $k=1$.  We observe that
  \begin{displaymath}
    v_4^1\overset{\beta_1^{-1}}{\longmapsto}
    v_1^1\overset{\alpha_1^{-1}}{\longmapsto}
    v_2^1\overset{\beta_1}{\longmapsto}
    v_3^1\overset{\alpha_1}{\longmapsto}
    v_0^1.
  \end{displaymath}
  Indeed, from the definition (see Section~\ref{subsec:deck-trans}) we
  see that the transformation $\alpha_1$ takes $v_3^1$ to
  $v_0^1$---think of the side $*\mathsf{a}_1 = (v_3^1, *a_1^1, *a_2^1,
  v_2^1)$ being mapped to the side $\mathsf{a}_1 = (v_0^1, a_1^1,
  a_2^1, v_1^1)$: the vertex $*a_1^1$ is mapped to $a_1^1$ and so
  $v_3^1$ is mapped to $v_0^1$.  We also see from
  Section~\ref{subsec:deck-trans} and
  Figure~\ref{fig:k-copy}\textsc{a} that $\beta_1$ maps $v_2^1$ to
  $v_3^1$, and so $v_0^1 = \alpha_1\beta_1\cdot v_2^1$.  A similar
  argument shows that $v_0^1 = \alpha_1\beta_1\alpha_1^{-1}\cdot
  v_1^1$ and that
  \begin{displaymath}
    v_0^1 = \alpha_1\beta_1\alpha_1^{-1}\beta_1^{-1}\cdot v_4^1 =
    \kappa_1\cdot v_4^1.
  \end{displaymath}
  
  Assuming the computations hold for $k-1$, we prove them for $k$.  In
  fact, most of the work is already done.  The same argument we used
  for the case $k=1$ shows that
  \begin{displaymath}
    v_4^k\overset{\beta_k^{-1}}{\longmapsto}
    v_1^k\overset{\alpha_k^{-1}}{\longmapsto}
    v_2^k\overset{\beta_k}{\longmapsto}
    v_3^k\overset{\alpha_k}{\longmapsto}
    v_0^k.
  \end{displaymath}
  The inductive hypothesis implies that
  \begin{displaymath}
    \kappa_{k-1} v_0^k = \kappa_{k-1} v_4^{k-1} = v_0^1,
  \end{displaymath}
  This ends the proof of the claim.

  These three claims prove that the labels in
  Figure~\ref{fig:k-copy}\textsc{a} are correct.

  (The labels in Figure~\ref{fig:m1-triangulated} also follow from
  these calculations, but may be obtained by more straightforward
  arguments because the generators of $\Gamma_1 \cong \Z^2$ may be
  regarded as shifts in the plane.)
\end{proof}

\begin{figure}
  \centerfloat
  \includegraphics{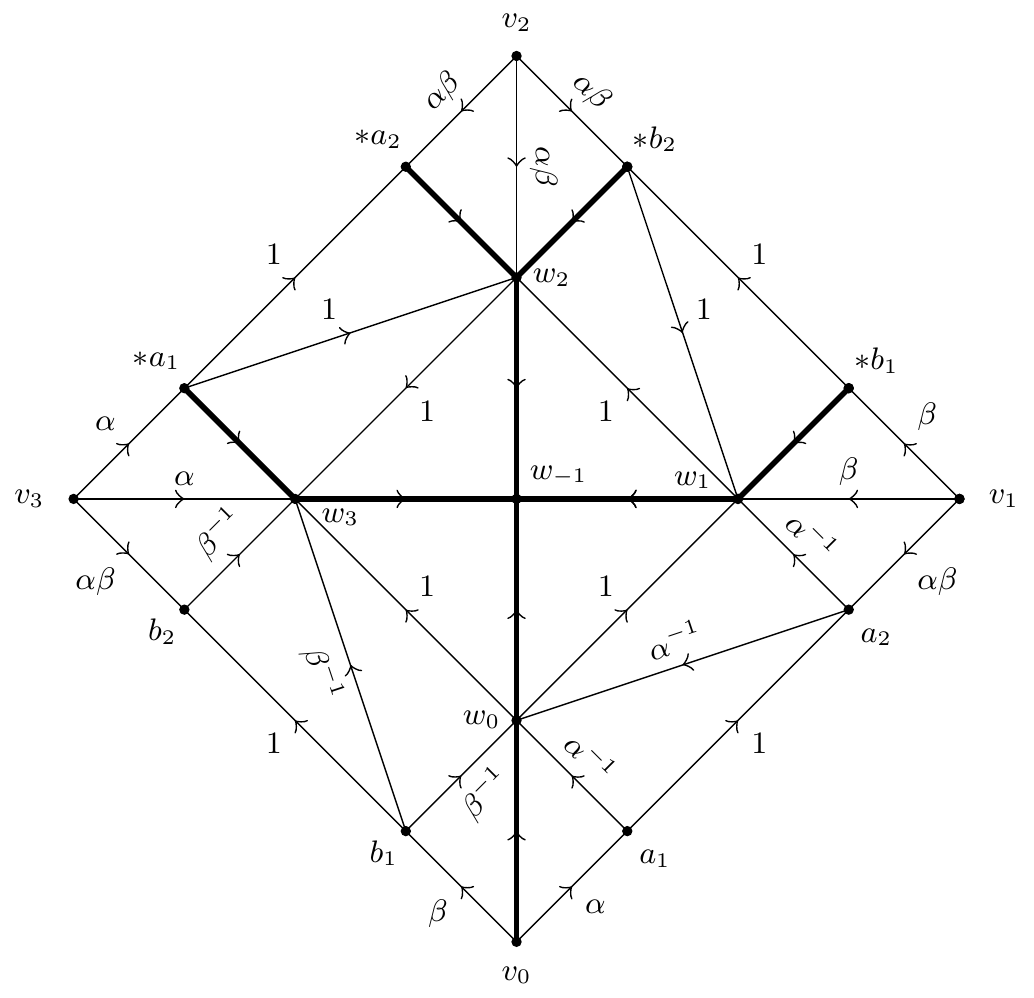}
  \caption{%
    The triangulation $\widetilde{\Lambda}_1$ of
    $\widetilde{\Sigma}_1$.  Edges are labeled with the group element
    associated with the loop they induce.  The spanning tree $T_1$ of
    $\Lambda_1$ (heavier edges) is ``rooted'' at $v_0$.}
  \label{fig:m1-triangulated}
\end{figure}

\begin{figure}
  \centerfloat
  \subfloat[]{\includegraphics{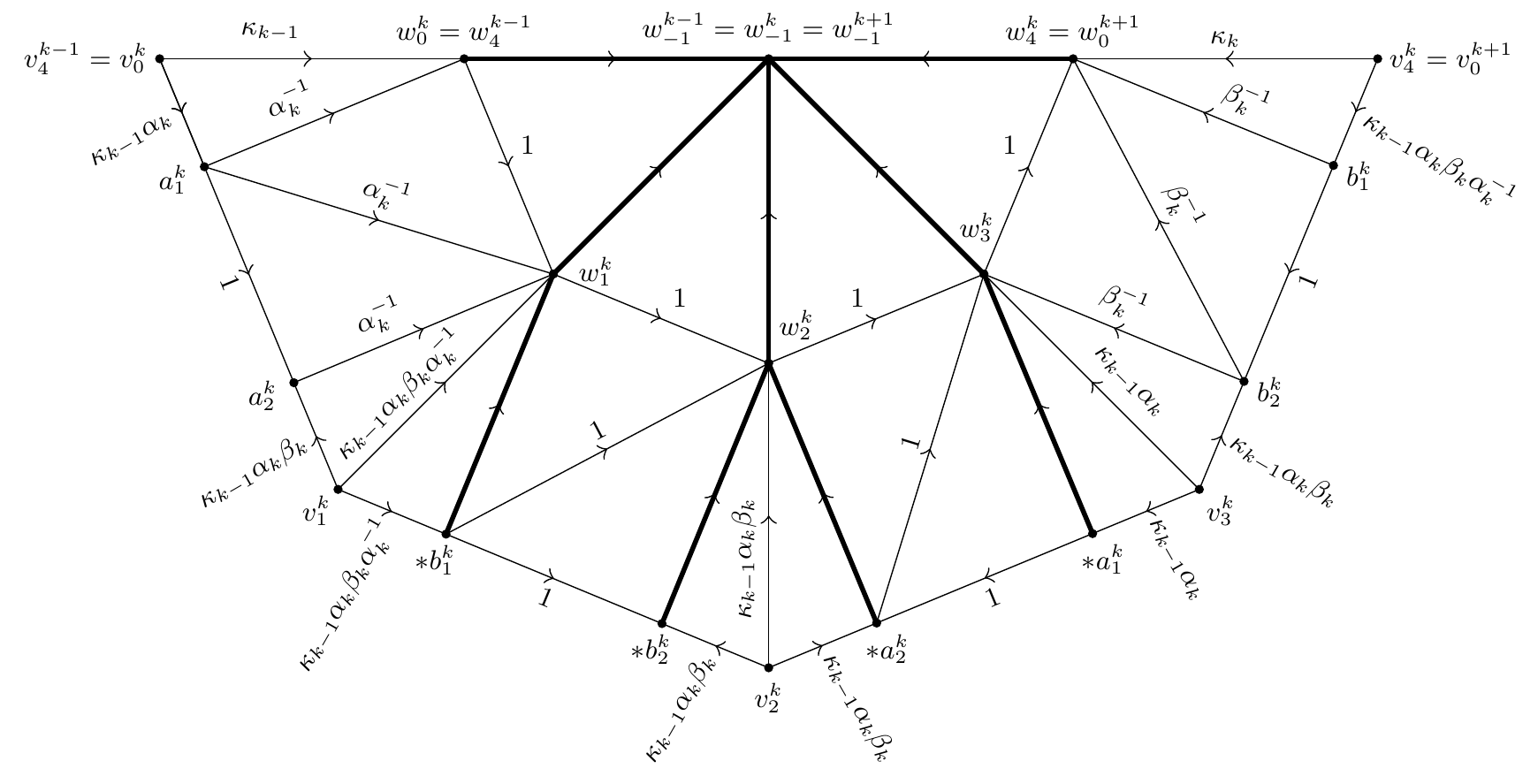}}\\[15pt]
  \subfloat[]{\includegraphics{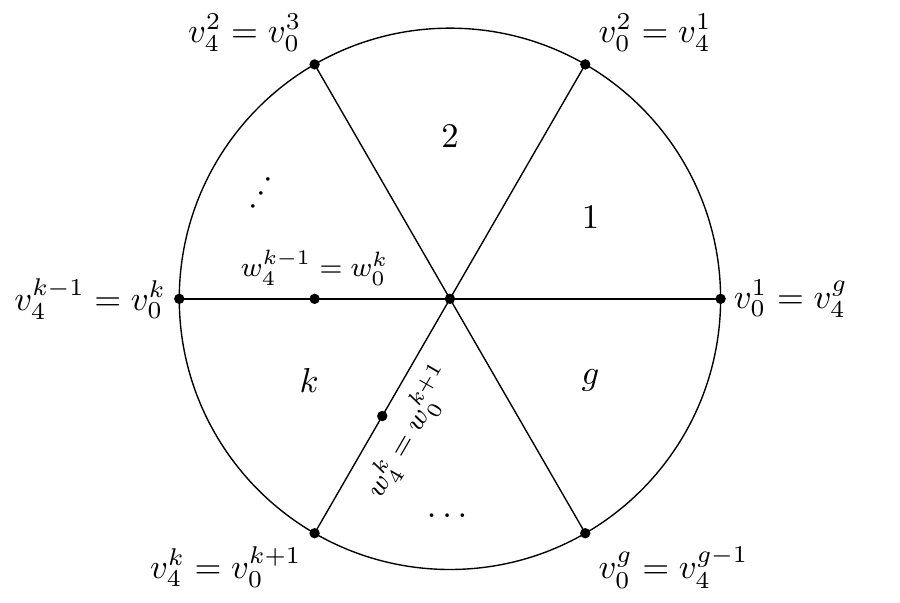}}
  \caption{%
    \textsc{(a)} The triangulation we use for $\widetilde{\Sigma}_1^k$,
    the $k$th copy of $\widetilde{\Sigma}_1$ (with a hole).  Every edge is
    labeled with the element of $\Gamma_g$ corresponding to the loop
    it induces.
    \textsc{(b)} How the simplicial complex $\widetilde{\Sigma}_g$ is
    defined.  The $k$th ``wedge'' is pictured in \textsc{(a)}.%
  }
  \label{fig:k-copy}
\end{figure}

\begin{notation}
\label{not:subset-F}
For $k\in \{1,\dots, g\}$, let
\begin{displaymath}
  {\cal F}_k = \{
  \alpha_k^{-1},\
  \beta_k^{-1},\
  \kappa_{k-1},\
  \kappa_{k-1} \alpha_k,\
  \kappa_{k-1} \alpha_k\beta_k,\
  \kappa_{k-1} \alpha_k\beta_k\alpha_k^{-1}
  \}.
\end{displaymath}
Notice that the set $\mathcal{F} = \{s_{ij}\}$ considered in
Section~\ref{subsec:s_ij} is equal to the union $\mathcal{F}_1\cup
\mathcal{F}_1^{-1}\cup \dots \cup \mathcal{F}_g\cup
\mathcal{F}_g^{-1}$ by Lemma~\ref{lem:edge-labels}.
\end{notation}

\subsection{Choosing quasi-representations}
We want to apply Theorem~\ref{thm:chern-class} using the labels
obtained in Lemma~\ref{lem:edge-labels} and some convenient choice of
a quasi-representation of $\Gamma_g$ in $U(A)$.  We begin by proving a
slightly stronger version Proposition~\ref{prop:existence-of-quasi},
which guarantees the existence of quasi-representations (under certain
conditions).  Let us set up some notation first.

For certain unitaries $u_1, v_1, \dots, u_g, v_g$ in $A$ we will need
to produce a quasi-representation $\pi$ satisfying
\begin{equation}
  \label{eq:7}
  \pi(\alpha_k) = u_k, \text{ and }
  \pi(\beta_k) = v_k\quad \forall k\in \{1, \dots, g\}.
\end{equation}
Write $\F_{2g} = \langle \hat{\alpha}_1, \hat{\beta}_1, \dots,
\hat{\alpha}_g, \hat{\beta}_g\rangle$ for the free group on $2g$
generators.  Let $q\colon \F_{2g}\to \Gamma_g$ and $\hat{\pi}\colon
\F_{2g}\to U(A)$ be the homomorphisms given by
\begin{displaymath}
  q(\hat{\alpha}_k) = \alpha_k,\quad
  q(\hat{\beta}_k) = \beta_k
\end{displaymath}
and
\begin{displaymath}
  \hat{\pi}(\hat{\alpha}_k) = u_k,\quad
  \hat{\pi}(\hat{\beta}_k) = v_k
  \end{displaymath}
  for all $k\in\{1, \dots, g\}$.  Notice that the kernel of $q$ is the
  normal subgroup generated by
\begin{displaymath}
  \hat{\kappa}_g := \prod_{k=1}^g [\hat{\alpha}_k, \hat{\beta}_k]
\end{displaymath}
and therefore consists of products of elements of the form
$\hat{\gamma} \hat{\kappa}_g^{\pm 1} \hat{\gamma}^{-1}$ where
$\hat{\gamma}\in \F_{2g}$.

Choose a set-theoretic section $s\colon \Gamma_g\to \F_{2g}$ of $q$
such that $s(1) = 1$,
\begin{displaymath}
  s(\alpha_k) = \hat{\alpha}_k,\quad
  \text{and}\quad
  s(\beta_k) = \hat{\beta}_k\quad \forall
  k\in \{1, \dots, g\}.
\end{displaymath}

\begin{lemma}
  \label{lem:existence-of-quasi-homs}
  For all $\epsilon > 0$ there exists $\delta(\epsilon) > 0$ such that
  if $A$ is a unital $C^*$-algebra and $u_1, v_1, \dots, u_g, v_g\in
  U(A)$ satisfy
  \begin{equation}
    \label{eq:8}
    \biggl\| \prod_{i=1}^g [u_i,v_i] - 1 \biggr\| <      
    \delta(\epsilon),
  \end{equation}
  then $\pi= \hat{\pi}\circ s$ (with $s$ as constructed above) is an
  $(\mathcal{F}, \epsilon)$-representation satisfying
  Eq.~(\ref{eq:7}).
\end{lemma}

This lemma obviously implies
Proposition~\ref{prop:existence-of-quasi}.

\begin{proof}
  We only need to check that $\pi$ is $(\mathcal{F},
  \epsilon)$-multiplicative.  Assume that Eq.~(\ref{eq:8}) holds for some
  $\delta$ in place of $\delta(\epsilon)$.

  Because $\hat{\pi}$ is a homomorphism, for all $\gamma, \gamma'\in
  \Gamma_g$ we have
  \begin{displaymath}
    \norm{\pi(\gamma)\pi(\gamma') - \pi(\gamma\gamma')} =
    \norm{\pi(\gamma)\pi(\gamma')\pi(\gamma\gamma')^* -
      1}
    = \big\|
    \hat{\pi}\big( s(\gamma)s(\gamma')
    s(\gamma\gamma')^{-1} \big) - 1
    \big\|.
  \end{displaymath}
  Now, $s(\gamma)s(\gamma')s(\gamma\gamma')^{-1}$ is in the kernel of
  $q$ and is therefore a product of the form
  \begin{displaymath}
    \prod_{i=1}^m\hat{\gamma}_i\hat{\kappa}_{g}^{\epsilon_i}\hat{\gamma}_i^{-1}
  \end{displaymath}
  where $m$ depends on $\gamma$ and $\gamma'$ and $\epsilon_i\in \{1,
  -1\}$.  Thus
  \begin{align*}  
    \norm{\pi(\gamma)\pi(\gamma') - \pi(\gamma\gamma')}
    &= \bigg\|\hat{\pi}\bigg(
    \prod_{i=1}^m\hat{\gamma}_i\hat{\kappa}_{g}^{\epsilon_i}\hat{\gamma}_i^{-1}
    \bigg)
    - 1\bigg\|\\
    &\leq \sum_{i=1}^m
    \norm{
      \hat{\pi}(\hat{\gamma}_i)
      \hat{\pi}(\hat{\kappa}_g)^{\epsilon_i}
      \hat{\pi}(\hat{\gamma}_i)^* - 1
    }\\
    & \leq m \bigg\| \prod_{i=1}^g [u_i,v_i] - 1 \bigg\|
    < m\delta.
  \end{align*}

  Since $\mathcal{F}$ is a finite set, there is a positive integer $M$
  such that if $\gamma, \gamma'\in \mathcal{F}$, then
  $s(\gamma)s(\gamma')s(\gamma\gamma')^{-1}$ is a product of at most
  $M$ elements of the form
  $\hat{\gamma}_i\hat{\kappa}_{g}^{\epsilon_i}\hat{\gamma}_i^{-1}$ as
  above.  It follows that $\pi$ is an $(\mathcal{F},
  M\delta)$-representation.  Choose $\delta(\epsilon) = \epsilon/M$.
\end{proof}

\begin{notation}
  \label{not:def-of-pi_0}
  Recall the set $\mathcal{F}_k$ defined in
  Notation~\ref{not:subset-F}. Let $s_0\colon \Gamma_g\to \F_{2g}$ be
  a set-theoretic section of $q$ such that
  \begin{gather*}
    s_0(\alpha_k^{\pm 1}) = \hat{\alpha}_k^{\pm 1},\quad
    s_0(\beta_k^{\pm 1}) = \hat{\beta}_k^{\pm 1},\quad
    s_0(\kappa_{k-1}) = \hat{\kappa}_{k-1},\\
    s_0(\kappa_{k-1}\alpha_k) = \hat{\kappa}_{k-1}\hat{\alpha}_k,\quad
    s_0(\kappa_{k-1}\alpha_k\beta_k)
    = \hat{\kappa}_{k-1}\hat{\alpha}_k\hat{\beta}_k
  \end{gather*}
  for all $k\in \{1, \dots, g\}$, and
  \begin{gather*}
    s_0(\kappa_{k-1}\alpha_k\beta_k\alpha_{k}^{-1})
    =     
    \hat{\kappa}_{k-1}\hat{\alpha}_k\hat{\beta}_k\hat{\alpha}_{k}^{-1}
  \end{gather*}
  for all $k\in \{1, \dots, g-1\}$.  That such a section exists
  follows from the fact that all the words in the list
  $\mathcal{F}_1\cup \cdots \cup \mathcal{F}_g\cup \{\alpha_1,
  \beta_1, \dots \alpha_g, \beta_g\}$ are distinct, with two
  exceptions: $\alpha_1 = \kappa_0 \alpha_1\in \mathcal{F}_1$ appears
  twice, as does $\beta_g =
  \kappa_{g-1}\alpha_g\beta_g\alpha_{g}^{-1}\in \mathcal{F}_g$.

  Define $\pi_0 = \hat{\pi}\circ s_0\colon \Gamma_g\to U(A)$.
\end{notation}

\begin{lemma}
  \label{lem:pi_0-mult-on-triangles}
  If $\langle x_i, x_j, x_k \rangle$ is any 2-simplex in $\Lambda_g$
  different from $\langle v_1^g, a_2^g, w_1^g \rangle$, then
  $\pi_0(s_{ik}) = \pi_0(s_{ij}) \pi_0(s_{jk})$.

  If $\langle x_i, x_j, x_k \rangle = \langle v_1^g, a_2^g, w_1^g
  \rangle$, then $\pi_0(s_{ik}) = v_g$ and
  \begin{displaymath}
    \pi_0(s_{ij}) \pi_0(s_{jk})
    = \bigg( \prod_{i=1}^g [u_i,v_i] \bigg) v_g.
  \end{displaymath}
\end{lemma}

\begin{proof}
  The definition of $s_0$ implies that the image under $s_0$ of any
  ``word'' in the list $\mathcal{F}_k$ is the word obtained by
  replacing $\alpha_k^{\pm 1}$ by $\hat{\alpha}_k^{\pm 1}$ and
  $\beta_k^{\pm 1}$ by $\hat{\beta}_k^{\pm 1}$, with one exception:
  the image of $\kappa_{g-1}\alpha_g\beta_g\alpha_{g}^{-1} = \beta_g$
  under $s_0$ is $\hat{\beta}_g$.

  This observation along with inspection of
  Figure~\ref{fig:k-copy}\textsc{a} shows that $s_0(s_{ik}) =
  s_0(s_{ij}) s_0(s_{jk})$ for every 2-simplex in $\Lambda_g$
  different from $\langle v_1^g, a_2^g, w_1^g \rangle$.  For instance,
  let $l\in \{ 1, \dots, g\}$ and consider the simplex
  \begin{displaymath}
    \langle v_0^l, a_1^l, w_0^l \rangle  = \langle x_i, x_j, x_k \rangle.
  \end{displaymath}
  The corresponding group elements are 
  \begin{align*}
    s_{ij} &= \kappa_{l-1} \alpha_l\\
    s_{jk} &= \alpha_l^{-1}, \quad\text{and }\\
    s_{ik} &= \kappa_{l-1}.
  \end{align*}
  Then
  \begin{displaymath}
    s_0(s_{ik}) = \hat{\kappa}_{l-1}
    = \hat{\kappa}_{l-1}\hat{\alpha}_l \cdot \hat{\alpha}_l^{-1}
    = s_0(\kappa_{l-1}\alpha_l) \cdot s_0(\alpha_l^{-1})
    = s_0(s_{ij})\cdot s_0(s_{jk}).
  \end{displaymath}
  The computations in all other 2-simplices but $\langle v_1^g, a_2^g,
  w_1^g \rangle$ are very similar.  For this exceptional simplex we
  get
  \begin{displaymath}
    s_0( s_{ik} ) = s_0( \kappa_{g-1}\alpha_g\beta_g\alpha_{g}^{-1} )
    = s_0( \beta_g ) = \hat{\beta}_g
  \end{displaymath}
  but
  \begin{displaymath}
    s_0( s_{ij} ) s_0( s_{jk} )
    = s_0 ( \kappa_{g-1}\alpha_g\beta_g ) s_0( \alpha_g^{-1} )
    = \hat{\kappa}_{g-1}\hat{\alpha}_g\hat{\beta}_g \hat{\alpha}_g^{-1}
    = \hat{\kappa}_g \hat{\beta}_g
  \end{displaymath}
  Since $\pi_0 = \hat{\pi} \circ s_0$ and $\hat{\pi}$ is a
  homomorphism, the lemma follows.
\end{proof}

Recall that we used Theorem~\ref{thm:index-theorem} to define $\omega
> 0$ in Section~\ref{subsec:push-forward-line}.

\begin{lemma}
  \label{lem:pi_0}
  If $0 < \epsilon < \omega$ and Eq.~(\ref{eq:8}) holds (so that $\pi_0$
  is an $(\mathcal{F}, \epsilon)$-representation), then
  \begin{displaymath}
    \tau \big( \pi_{0\, \sharp}(\mu[\Sigma_g]) \big)
    = \frac{1}{2\pi i} \tau\bigg(
      \log\bigg( \prod_{i=1}^g [u_i,v_i] \bigg)
      \bigg).
  \end{displaymath}
\end{lemma}

\begin{proof}
  We apply Theorem~\ref{thm:chern-class}.  For each simplex $\langle
  x_i, x_j, x_k\rangle$ we compute $\widetilde{\Delta}_\tau(\xi)$
  where $\xi_\sigma$ is the path
  \begin{displaymath}
    \xi_\sigma(t) = (1-t)\pi(s_{ik}) + t\pi(s_{ij})\pi(s_{jk}),\quad
    t\in[0,1].
  \end{displaymath}

  Observe that the value of $\widetilde{\Delta}_\tau$ on a constant
  path is 0.  Lemma~\ref{lem:pi_0-mult-on-triangles} implies that
  there is only one 2-simplex $\sigma$ such that $\xi_\sigma$ is not
  constant: $\sigma_0 = \langle v_1^g, a_2^g, w_1^g \rangle$.  By
  Lemma~\ref{lem:pi_0-mult-on-triangles} it yields the linear path
  $\xi_{\sigma_0}$ from $v_g$ to
  \begin{displaymath}
    \bigg( \prod_{i=1}^g [u_i,v_i] \bigg) v_g.
  \end{displaymath}
  Using Lemma~\ref{lem:de-la-harpe} we obtain
  \begin{displaymath}
    \widetilde{\Delta}_\tau(\xi_{\sigma_0}) = \frac{1}{2\pi i} \tau\bigg(
      \log\bigg( \prod_{i=1}^g [u_i,v_i] \bigg)
      \bigg).
  \end{displaymath}
  Finally, Theorem~\ref{thm:chern-class} implies
  \begin{displaymath}
    \tau \big( \pi_{0\, \sharp}(\mu[\Sigma_g]) \big)
    = (-1)^{s(\sigma_0)}\frac{1}{2\pi i} \tau\bigg(
    \log\bigg( \prod_{i=1}^g [u_i,v_i] \bigg)
    \bigg)
  \end{displaymath}
  where the sign $(-1)^{s(\sigma_0)}$ depends on the the orientation
  $[\Sigma_g]$.  The standard orientation on $\Sigma_g$ gives
  ${s(\sigma_0)} = 1$.
\end{proof}

By putting these lemmas together we can prove
Theorem~\ref{thm:Mg-formula}.

\begin{proof}[Proof of Theorem~\ref{thm:Mg-formula}]
  Recall that the statement of the theorem fixes a positive integer
  $g$ and idempotents $q_0$ and $q_1$ in some matrix algebra over
  $\ell^1(\Gamma_g)$ such that $\mu[\Sigma_g] = [q_0] - [q_1]\in
  K_0(\ell^1(\Gamma_g))$.

  Let $\mathcal{F}_0$ be the finite set $\{s_{ij}\}$ defined in
  Section~\ref{subsec:s_ij} and described explicitly in
  Notation~\ref{not:subset-F}.  Theorem~\ref{thm:index-theorem}
  provides an $\omega > 0$ so small that if $\pi\colon \Gamma_g\to
  U(A)$ is an $(\mathcal{F}_0, \omega)$-representation, then
  $\pi_\sharp(\mu[\Sigma_g]) := \pi_\sharp(q_0) - \pi_\sharp(q_1)$ is
  defined and
  \begin{displaymath}
    \tau\big( \pi_\sharp(\mu[\Sigma_g]) \big) =
    \langle \ch_\tau (\ell_\pi), [\Sigma_g] \rangle.
  \end{displaymath}

  By setting $u_i := \pi(\alpha_i)$ and $v_i := \pi(\beta_i)$ for all
  $i\in \{1, \dots, g\}$, we see that such a quasi-representation
  $\pi$ may be used to define a quasi-representation $\pi_0$ as in
  Section~\ref{not:def-of-pi_0}.  The more multiplicative $\pi$ is on
  $\mathcal{F}_0$, the smaller the quantity
  \begin{displaymath}
    \biggl\| \prod_{i=1}^g [u_i,v_i] - 1 \biggr\|
  \end{displaymath}
  is.  Lemma~\ref{lem:existence-of-quasi-homs} shows that by making
  this quantity smaller we can make $\pi_0$ more multiplicative on
  $\mathcal{F}_0$.  Therefore, because $\pi$ and $\pi_0$ agree on the
  generators of $\Gamma_g$, there exists an $0 < \epsilon_0 < \omega$
  so small that if $\pi$ is an $(\mathcal{F}_0,
  \epsilon_0)$-representation, then $\pi_\sharp$ and $\pi_{0\,
    \sharp}$ agree on $\{q_0, q_1\}\subset K_0(\ell^1(\Gamma_g))$.

  Finally,
  \begin{displaymath}
    \tau\big( \pi_\sharp(\mu[\Sigma_g]) \big)
    = \tau\big( \pi_{0\, \sharp}(\mu[\Sigma_g]) \big)
    = \frac{1}{2\pi i} \tau\bigg(
    \log\bigg( \prod_{i=1}^g [u_i,v_i] \bigg)
    \bigg)
  \end{displaymath}
  by Lemma~\ref{lem:pi_0}.
\end{proof}

\bibliographystyle{amsalpha}

\providecommand{\bysame}{\leavevmode\hbox to3em{\hrulefill}\thinspace}
\providecommand{\MR}{\relax\ifhmode\unskip\space\fi MR }
\providecommand{\MRhref}[2]{%
  \href{http://www.ams.org/mathscinet-getitem?mr=#1}{#2}
}
\providecommand{\href}[2]{#2}

\end{document}